\documentclass[12pt,a4paper]{amsart}

\usepackage[text={400pt,650pt},centering]{geometry}

\usepackage{color}
\usepackage{esint,amssymb}
\usepackage{graphicx}
\usepackage{MnSymbol}
\usepackage{mathtools}
\usepackage[colorlinks=true, pdfstartview=FitV, linkcolor=blue, citecolor=blue, urlcolor=blue,pagebackref=false]{hyperref}
\usepackage{microtype}

\newtheorem{theorem}{Theorem}
\newtheorem{proposition}[theorem]{Proposition}
\newtheorem{lemma}[theorem]{Lemma}
\newtheorem{corollary}[theorem]{Corollary}

\theoremstyle{definition}
\newtheorem{remark}[theorem]{Remark}

\newtheorem{definition}[theorem]{Definition}

\newcommand{\cref}[1]{Corollary~\ref{c.#1}}

\numberwithin{equation}{section}
\numberwithin{theorem}{section}

\newcommand{\R}{\mathbb{R}}

\renewcommand{\subset}{\subseteq}

\DeclareMathOperator*{\osc}{osc}

\DeclareMathOperator{\supp}{supp}

\renewcommand{\tilde}{\widetilde}

\newcommand{\intav}[1]{\mathchoice {\mathop{\vrule width 6pt height 3 pt depth  -2.5pt
\kern -8pt \intop}\nolimits_{\kern -6pt#1}} {\mathop{\vrule width
5pt height 3  pt depth -2.6pt \kern -6pt \intop}\nolimits_{#1}}
{\mathop{\vrule width 5pt height 3 pt depth -2.6pt \kern -6pt
\intop}\nolimits_{#1}} {\mathop{\vrule width 5pt height 3 pt depth
-2.6pt \kern -6pt \intop}\nolimits_{#1}}}

\begin{document}

\title[Fully nonlinear elliptic transmission problems ]{High-Order regularity for Fully nonlinear elliptic transmission problems under weak convexity assumption }

\begin{abstract}
This paper studies, in the spirit of \cite{MPA},  Schauder theory to transmission problems modelled by fully nonlinear uniformly elliptic  equations of second order. We focus on operators $F$ that fails to be concave or convex in the space of symmetric matrices.  In a first scenario, it is considered that  $F$ enjoys a small ellipticity aperture.  In our second case, we study regularity results  where the convexity of the superlevel (or sublevel) sets is verified, implying that the operator F is quasiconcave (or quasiconvex).
\end{abstract}

\author[G.C.Ricarte]{G. C. Ricarte}
\address[G.C.Ricarte]{Universidade Federal Cear\'{a}. Department of Mathematics. Fortaleza, CE-Brazil 60455-760. }
\email{ricarte@mat.ufc.br}

\author[C.S. Barroso]{C. S. Barroso}
\address[C.S.Barroso]{Universidade Federal Cear\'{a}. Department of Mathematics. Fortaleza, CE-Brazil 60455-760. }
\email{cleonbar@mat.ufc.br}

\author[L. S. Tavares ]{L. S. Tavares}
\address[L. S. Tavares ]{Universidade Federal do Cariri. Centro de Ciências e Tecnologia. Juazeiro do Norte, CE-Brazil 63048-080.}
\email{leandro.tavares@ufca.edu.br}

\date{\today}

\maketitle

\section{Introduction}

In the past few years there has been increasing interest in so called transmission problems. These problems appear, for example, in thermal theory and electromagnetic conductivity, composite materials.  The transmission problems were initially studied in the late 30’s  by Picone \cite{Picone} addressing the concept within the realm of materials science.  In general, problems of this type can be described as follows. There are given two domains $\Omega^{+}, \Omega^{-} \subset \mathbb{R}^n$, space which have a portion $\Gamma$ of their boundaries in common. A boundary value problem is then posed for each domain with the boundary conditions on $\Gamma$  and involving the solutions of both problems. As a sample, let $\Omega^{+}$ and $\Omega^{-}$ domains with $\partial \Omega^{+}$ and $\partial \Omega^{-}$, respectively. The problems consists in to consider qualitative and quantitative aspects of a function that satisfy   $u$ that satisfy  $u^{+}=u \big |_{\Omega^+},$ $u^{-}= u \big |_{\Omega^{-}}$, respectively, for which $u^{+}_{\nu} - u^{-}_{\nu}=g$ on $\Gamma$. More precisely, 
\begin{equation}\label{transp}
\left\{
\begin{array}{rclcl}
 \Delta u^{+} &=& f^{+}(x)& \mbox{in} &   \Omega^{+}  \\
 \Delta u^{-} &=& f^{-} (x) &\mbox{in}& \Omega^{-} \\
 u^+_{\nu} - u^{-}_{\nu}&=& g (x) &\mbox{on}& \Gamma,
\end{array}
\right.
\end{equation}
The main contribution in \cite{Picone} lies in the uniqueness of solutions for a wide class of  transmission problems related to mathematical models in elasticity theory. It is worth noting that the first  rigorous contribution regarding existence results traces back to Lions \cite{lions}. Furthermore, both contributions of Picone and Lions have inspired subsequent developments, including the ones by Stampacchia \cite{stampacchia}, Campanato \cite{camp1, camp2} and Schechter \cite{schechter}. Additionally, we reference the work of Borsuk \cite{borsuk}, where an extensive list  of significant contributions related to fixed transmission problems is provided alongside and a comprehensive exposition of their mathematical theory and aspects.

In this paper we consider the following general transmission problem: Given two functions $f^{+}$, $f^{-}$, to find functions $u^{\pm} \in C^{2,\alpha}$ such that 
\begin{equation}\label{E1}
\left\{
\begin{array}{rclcl}
 F(D^2u) &=& f^{+}(x)& \mbox{in} &   \Omega^{+}=B_1 \cap \{x_n > \psi(x')\}  \\
 F(D^2 u) &=& f^{-}(x) &\mbox{in}& \Omega^{-} = B_1 \cap \{x_n < \psi(x')\}\\
 u^+_{\nu} - u^{-}_{\nu}&=& g(x) &\mbox{on}& \Gamma.
\end{array}
\right.
\end{equation}
  Here $u \in C(B_1)$ and $u^{\pm} = u \big|_{\overline{\Omega^{\pm}}}$. The interface $\Gamma$ is given by the graph of a function $\psi : \mathbb{R}^{n-1} \rightarrow \mathbb{R}$ with unit normal vector $\nu$ pointing towards $\Omega^+$, and $u^{\pm}_{\nu}$ denote the $\nu$-directional derivatives of $u^{\pm}$.  Let $\textrm{Sym}(n)$ be the set of symmetric matrices of size $n \times n$. We assume that $F : \textrm{Sym}(n)  \rightarrow \mathbb{R}$ satisfies
\begin{equation}\tag{{\bf \text{\textcolor{blue}{Unif. Ellip.}}}}\label{Unif.Ellip.}
  \hskip -1.5cm \lambda\|N\|\leq F(M+N)-F(M) \leq \Lambda \|N\|
\end{equation}
 for every $M,N \in \textit{Sym}(n)$ with $N \ge 0$. For simplicity, we assume $F(0)=0$. As usual, $\mathcal{E}$$(\lambda, \Lambda)$ denotes the class of all such operators.  For $0 < \lambda \le \Lambda$, set
$$\mathcal{E}(\lambda,\Lambda) = \{F : \textrm{Sym}(n)  \rightarrow \mathbb{R} : \,\, F \,\, \textrm{satisfies} \,\, \eqref{Unif.Ellip.} \}$$

Transmission problems in the context of viscosity solutions have been studied recently by Soria-Carro and Stinga \cite{MPA}. In their fine study, which includes problems like \eqref{E1}, they achieved significant results, including the existence of solutions through a Alexandrov-Bakelman-Pucci type maximum principle (see \cite[Section 2]{MPA} for a more general statement), uniqueness of the solution, a comparison principle, and regularity of solutions of flat interface transmission problems. As for regularity results, they obtained $C^{0,\alpha}$, $C^{1,\alpha}$, and $C^{2,\alpha}$ regularity of viscosity solutions up to the transmission surface for the case of curved interfaces. An essential aspect to highlight is that they proved global $C^{2,\alpha}$ estimates for some $\alpha \in (0,1)$ under the convexity hypothesis on  $F$, which do not necessarily cover the whole of $(0,1)$. 

Inspired by \cite{JV1,MPA} and under a different approach, the goal of this paper is twofold. The first one is to establish $C^{2,\alpha}$ estimates with $\alpha$ running throughout $(0,1)$ for viscosity solutions of \eqref{E1}, complementing and improving \cite[Theorem 1.3]{MPA}. More precisely, we explore the so-called \textit{geometric tangential equation} which is usually used to study approximation operators under Cordes-Nirenberg’s conditions in several contexts of
regularity. Here we will assume  that the ellipticity constants $(\lambda,\Lambda)$ do not deviate significantly, that is, when the quantity $\mathfrak{e} \colon= 1-\frac{\lambda}{\Lambda}$ is sufficiently small (\textcolor{black}{Theorem} \ref{BMO}).  

Of particular interest, we point out that our approach covers the class of \textit{Isaac's type equations}
\begin{equation} \label{Isaac1}
\left\{
\begin{array}{rclcl}
\displaystyle \sup_{\beta \in \mathcal{B}} \inf_{\gamma \in \mathcal{A}} \left(\mathcal{L}_{\gamma \beta} u - h_{\gamma \beta}(x) \right)&=& f^{\pm}(x)& \mbox{in} &   \Omega^{\pm}  \\
 u^+_{\nu} - u^{-}_{\nu}&=& g (x) &\mbox{on}& \Gamma,
\end{array}
\right.
\end{equation}
where $h_{\gamma \beta} : B_1 \rightarrow \mathbb{R}$ and $f^{\pm} : \Omega^{\pm} \rightarrow \mathbb{R}$ are H\"{o}lder continuous and $\mathcal{L}_{\gamma \beta} u \colon= a^{ij}_{\gamma \beta}(x) \partial_{ij} u $ is a family of elliptic operators. This class of operators appears naturally in stochastic control and in differential games theory, for example, two-player and zero-sum differential games (see  \cite{Iss} and \cite{AFD}). Note that this is an example of a non-convex/non-concave equation; hence, Theorem 1.3 in \cite{MPA}  does not cover \eqref{Isaac1} (see e.g.  Corollary \ref{Isaac}) for $C^{2,\alpha}$ estimates. However, the best result obtained in \cite{MPA} provides $C^{1,\alpha}$ estimates. 

In our second goal, we consider scenarios where the convexity of the superlevel (or sublevel) sets is verified, implying that the operator $F$ is quasiconcave (or quasiconvex). Alternatively, we also impose certain asymptotic concavity properties, such as e.g. the condition that $F = F(M)$ behaves as a concave (or convex, or ``close to linear") function when $M$ becomes large. In light of these considerations, we establish Theorem \ref{BMOQ} that furnishes $C^{2,\alpha}$ estimates for  $\alpha \in (0,\alpha_0)$, where $\alpha_0 \in (0,1)$is as in \eqref{Hom1}.

A family of operators for which our second result (Theorem \ref{BMOQ}) applies concerns the \textit{special Lagrangian equation} which
has the form
\begin{equation} \label{Lag}
F(D^2 u) = \sum_{j=1}^{n} \textrm{arctg} (\lambda_j) =\Theta,
\end{equation}
where $\lambda_j$ are the eigenvalues of Hessian matrix $D^2u$.  Equation \eqref{Lag} originates in the special Lagrangian geometry by Harvey-Lawson. The Lagrangian graph $(x, Du(x)) \subset \mathbb{R}^n \times \mathbb{R}^n$ is called spacial when the argument of the complex number $(1+ i \lambda_1) \cdot \ldots \cdot (1+i \lambda_n)$ or the phase is constant $\Theta$, and it is special if and only if $(x,Du(x))$ is a minimal surface in $\mathbb{R}^n \times \mathbb{R}^n$.  Therefore, using our Schauder type results Theorem \ref{BMOQ}, we deduce that solutions of special Lagrangian equations under the Yuan's assumption  $|f^{\pm}(x)| \ge \frac{\pi}{2}(n-2)$ are $C^{2,\alpha}$ for any dimension.

It is worth mentioning that in our results we only deal with constant coefficients. However, similar estimates can be derived for equations with H\"{o}lder continuous coefficients (see Section \ref{SecCasoGeral} for more details).

\subsection{Historic overview and further motivations}
In the $50$s  H. Cordes and L. Nirenberg, independently, established several results on linear elliptic problems in non-divergence form. Roughly speaking, based on perturbation arguments, they showed for a bounded solution of the problem
 \begin{equation} \label{CH}
 	a_{ij}(x) \partial_{ij} u =f(X) \quad \textrm{in} \quad B_1
 \end{equation}
with $(a_{ij}(x))$ a symmetric matrix and uniformly elliptic, and a given $\delta >0$, that there exist $0 < \alpha(n,\lambda,\Lambda, \delta) < 1$ such that if 
 $
 	|a_{ij}(x) - a_{ij}(x_0)| < \delta
$
then the solutions of \eqref{CH} are $C^{1,\alpha}_{loc}(B_1)$. 
 
Let's recall that in the context of classical solutions for  
   \begin{equation} \label{T1}
  F(D^2 u) = 0 \quad \textrm{in} \quad B_1
  \end{equation}
where $F$ is only assumed to be either uniformly elliptic or parabolic, viscosity solutions may fail to be smooth and
the best known regularity under these sole assumptions is $C^{1,\alpha}$ regularity. For example, Nadirashvilli and Vladut \cite{Vla} showed the existence of nonclassical viscosity solutions to fully nonlinear elliptic equations in dimension $12$. In fact, such solutions are not even $C^{1,1}$ \cite{Vla1} and \cite{Vla2}. Thus, a relevant problem in the theory is to determine some structural conditions on $F$ other than the uniform ellipticity (between concavity and convexity, and without further hypotheses) that guarantee higher order $C^{2,\alpha}$-regularity. With this issue in mind, another reasonable question is: 
 \begin{center}
 {\sl Which assumptions on $F$ guarantee that solutions of  $F(D^2u)=f(x)$ in $B_1$, for $f$ in a suitable function space, are classical?  }
 \end{center}
The first relevant regularity result for \eqref{T1} is due to L. Nirenberg \cite{Nir} who derived \textit{a priori} $C^{2,\alpha}$ estimates in dimension $2$. Through the journey of finding $C^2$-solutions for fully nonlinear equations, under concavity or convexity assumption, the $C^{2,\alpha}$-regularity result ($\alpha\in (0,1)$) of Evans \cite{Ev} and Krylov \cite{Kr} is groundbreaking. The second major chapter in the  theory regards Armstrong's, Silvestre and Smart's estimates in \cite{ASS}, where they proved under smoothness on $F$ that solutions of \eqref{T1} are partially regular, that is, there exist a closed set $\Sigma \subset B_1$ and a universal constant $\varepsilon>0$ such that $u \in C^{2, \alpha}(B_1\backslash \Sigma)$ with $\mathcal{H}^{n-\varepsilon}(\Sigma)=0$.  
  Important results in this direction can be found in \cite{WN}, where Cordes-Landis ellipticity type conditions are assumed. More precisely, fixed any $\alpha \in (0,1)$, there exist a $\delta>0$ such that, if $\frac{\lambda}{\Lambda}-1 < \delta$ and $f \in C^{\alpha}(0)$ then $u \in C^{2,\alpha}(0)$. We also mention the work of Goffi \cite{Al} where uniform ellipticity is combined with convexity of the superlevel (sublevel) sets to reach $C^{2,\alpha}$ regularity.

In the context of regularity theory for transmission problems, the recent literature concerning the problem \eqref{E1} also includes contribution of Cafarelli, Soria-Carro, and Stinga \cite{cscs}. Their research focuses on attaining $C^{1,\alpha}$ regularity up to the transmission surface for distributional solutions of  the Dirichlet problem associated with \eqref{transp}, under the assumption that the transmission surface possesses $C^{1,\alpha}$ regularity. Their notable achievements were made possible via the development of a novel geometric stability argument based in the mean value property. In \cite{dong} Dong provided an alternative proof of the $C^{1,\alpha}$-regularity result obtained in \cite{cscs}. Notably, his proof method differs from that in \cite{cscs} and embraces more generalized elliptic systems featuring variable coefficients. Furthermore, extensions to $C^{1,\text{Dini}}$ interfaces and domains encompassing multiple sub-domains were also explored.

With regarding to degenerate operators, we mention the work by  Bianca, Pimentel, and Urbano \cite{bpu} where they analyzed \eqref{transp} in the framework of Orlicz spaces. Natural challenges (as the lack of representation operator formulas and the degenerate characteristics of the diffusion process) arose from their study. In the scenario of bounded interface data, the local boundedness of weak solutions was proved. Also an estimate for their gradient in BMO spaces was obtained via Campanato type arguments, furnishing Log-Lipschitz regularity across the transmission interface. Furthermore, by relaxing the data constraints, local Hölder continuity for the solutions was also obtained.

Meanwhile, for fully nonlinear elliptic transmission, there have some  noteworthy preceding results.   Very recently in Soria-Carro and Stinga  \cite{MPA} the authors developed a robust study concerning the  regularity theory of viscosity solutions to transmission problems for fully nonlinear second order uniformly elliptic equations.  The authors of the mentioned reference obtained  $C^{0,\alpha}, C^{1,\alpha}$ and $C^{2,\alpha}$ estimates for viscosity solutions to \eqref{E1} for some $\alpha \in (0,1).$ The proof in \cite{MPA} relies on a crucial global  estimate for homogenous flat interface transmission problem 
\begin{equation}\label{E_H}
\left\{
\begin{array}{rclcl}
 F(D^2u) &=& 0& \mbox{in} &   B^{\pm}_1 \\
 u^+_{x_n} - u^{-}_{x_n}&=& 0 &\mbox{on}& T=B_1 \cap \{x_n=0\},
\end{array}
\right.
\end{equation}
where $F$ is a convex operator, combined with a certain geometric oscillation estimate which has its roots in the seminal paper of Caffarelli \cite{Caff1}.  In our case, we will replace Lemma 4.19 in \cite{MPA} with Proposition \ref{Grad} in the first case and Theorem 4.16 in \cite{MPA} with Theorem \ref{flat} in the second case.

\section{Preliminaries and statements of the main results}
 
 For a given $r >0$ and $x \in \mathbb{R}^n$, we denote by $B_r(x) \subset \mathbb{R}^n$ the ball of radius $r$ centered at $x=(x',x_n)$, where $x'=(x_1,x_2,\ldots, x_{n-1})$ and $B'_r(x) \subset \mathbb{R}^{n-1}$ the ball of radius $r$ centered at $x'$. We write $B_r = B_r(0)$ and $B^{\pm}_r = B_r \cap \mathbb{R}^n_{\pm}$. Also, we write $T_r \colon= \{(x',0) \in \mathbb{R}^{n-1} : |x'| < r\}$ and $T_r(x_0) \colon= T_r + x'_0$ where $x'_0 \in \mathbb{R}^{n-1}$. Given $r>0$, we write $\Omega^{\pm}_r = \Omega^{\pm} \cap B_r$.  $\nabla' $ denotes the gradient in the variables $x'$, $D^2_{x'}$ denotes the Hessian in the variables $x'$.  In what follows, $  \textrm{USC}(B_1)$ and $ \textrm{LSC}(B1)$  are the sets of upper semicontinuous and lower semicontinuous functions on $B_1$, respectively.  For $0 < \lambda \le \Lambda$, we define the Pucci's extremal operators $\mathcal{M}^{\pm}_{\lambda, \Lambda}$ as
$$
\mathcal{M}^{+}_{\lambda,\Lambda}(M) =  \lambda \cdot \sum_{e_i <0} e_i + \Lambda \cdot \sum_{e_i >0} e_i \quad \textrm{and} \quad \mathcal{M}^{-}_{\lambda,\Lambda}(M) =  \lambda \cdot \sum_{e_i >0} e_i + \Lambda \cdot \sum_{e_i <0} e_i 
$$
where $\{e_i : 1 \le i \le n\}$ denote the eigenvalues of $M$. 

Since our focus is on regularity for fully nonlinear elliptic transmission problems \eqref{E1},  we introduce the to  appropriate notion of weak solutions.
\begin{definition}[{\bf Viscosity solutions}]\label{VSLp}Consider a function  $F \in \mathcal{E}$$(\lambda,\Lambda)$.  We say that  $u \in \textrm{USC}(B_1) $ is  a viscosity solution to the transmission problem \eqref{E1}    if the following conditions are true:
\begin{enumerate}
\item[a)] If $x_0 \in \Omega^{\pm}$ and $\phi \in C^2(B_{\delta}(x_0))$ is  touching the function $u$ by above at $x_0$ in $B_1$, then
$$
F\left(D^2 \phi(x_0)\right)  \geq f^{\pm}(x_0)
$$
and if $x_0 \in \Gamma$, $\phi \in C^2(\overline{B^+_{\delta}(X_0)}) \cap C^2(\overline{B^{-}_{\delta}(x_0)})$, then $\phi^+_{\nu}(x_0) - \phi^{-}_{\nu}(x_0) \ge g(x_0)$  at $x_0 \in \Gamma$.

\item[b)] If $x_0 \in \Omega^{\pm}$ and $\phi \in C^2(B_{\delta}(x_0))$ is touching the function  $u$ by below at $x_0$ in $B_1$, then
$$
F\left(D^2 \phi(x_0)\right)  \leq f^{\pm}(x_0)
$$
and if $x_0 \in \Gamma$, $\phi \in C^2(\overline{B^+_{\delta}(X_0)}) \cap C^2(\overline{B^{-}_{\delta}(x_0)})$, then $\phi^+_{\nu}(x_0) - \phi^{-}_{\nu}(x_0) \le g(x_0)$  at $x_0 \in \Gamma$.
\end{enumerate}
\end{definition}

\begin{definition}
 We denote by  $\overline{\mathcal{S}}\left(\lambda,\Lambda, f^{\pm}\right)$ and $\underline{S}\left(\lambda,\Lambda, f^{\pm}\right)$  the sets of all continuous functions $u$ that satisfy $\mathcal{M}^{+}_{\lambda, \Lambda}(D^2 u) \ge f^{\pm}$,  $\mathcal{M}^{-}_{\lambda,\Lambda}(D^2 u) \le f^{\pm}$ in the viscosity sense respectively. We also denote
 $$
 	\mathcal{S}\left(\lambda, \Lambda,f^{\pm}\right) \colon=  \overline{\mathcal{S}}\left(\lambda, \Lambda,f^{\pm}\right) \cap \underline{\mathcal{S}}\left(\lambda, \Lambda, f^{\pm}\right)$$
 and	
 	$$
 \mathcal{S}^{\star}\left(\lambda, \Lambda, f^{\pm}\right) \colon=  \overline{\mathcal{S}}\left(\lambda, \Lambda,|f^{\pm}|\right) \cap \underline{\mathcal{S}}\left(\lambda, \Lambda, -|f^{\pm}|\right).$$
  
 \end{definition}
 
 We collect some tools from transmission problems, referring to \cite{MPA} for more details.   An important piece of information which we need in our article concerns the notion of stability of viscosity solutions, i.e., the limit of a sequence of viscosity solutions turns
out to be a viscosity solution of the corresponding limiting equation. We refer to the
following Lemma, whose proof can be found in \cite[Lemma 5.1]{MPA}.
\begin{lemma}[{\bf Stability Lemma}]\label{Est}
For $j\in\mathbb{N}$ let $\Gamma_j \in C^2$ and assume that $u_{j} \in C(B_1)$ is a viscosity solution to the problem
$$
	\left\{
		\begin{array}{rclcl}
			F_j(D^2 u_j) = f^{\pm}_{j} &\mbox{in}&   \Omega^{\pm}_j \\
			(u^{+}_j)_{\nu} - (u^{-}_j)_{\nu}=g_j  &\mbox{on}&  \Gamma_j,
		\end{array}
		\right.
$$
where $\Gamma_j = B_1 \cap \{x_n = \psi_j(x')\}$ for $\psi_j \in C^2(B'_1)$, $f^{\pm}_j \in C(\Omega^{\pm}_j \cup \Gamma_j)$, and $g_j \in C(\Gamma_j)$, for $g \ge 1$. Suppose also that there are continuous functions $u, f^{\pm}, g$, and elliptic operators $F^{\pm} \in \mathcal{E}$$(\lambda,\Lambda)$ such that
\begin{enumerate}
\item[(i)] $F_{j} \to F$ uniformly on compact subsets of $\textrm{Sym}(n)$;
\item[(ii)] $u_j \to u$ uniformly on compact subsets of $B_1$;
\item[(iii)] $\|f^{\pm}_j - f^{\pm}\|_{L^{\infty}(\Omega^{\pm}_j)} \to 0$;
\item[(iv)] $\|g_j - g\|_{L^{\infty}(\Gamma_j)} = \sup_{x'\in B'_1} |g_j(x', \psi_j(x')) -g(x',0)| \to 0$;
\item[(v)] $\Gamma_j \to T$ in $C^2$ in the sense that $\|\psi_j\|_{C^2(B'_1)} \to 0$.
\end{enumerate}
Then $u \in C(B_1)$ is a viscosity solution to the problem
$$
\left\{
		\begin{array}{rclcl}
			F(D^2 u) = f^{\pm} &\mbox{in}&   B^{\pm}_1 , \\
			u^{+}_{\nu} - u^{-}_{\nu}=g  &\mbox{on}& T.
		\end{array}
		\right.
$$
\end{lemma}

The fundamental result  to obtain  regularity results  of viscosity solutions to the transmission problem described in  \eqref{E1} will be a ABP estimate. For a proof of such result   see for instance \cite[Theorem 2.1]{MPA}
\begin{lemma}[{\bf ABP estimate}]\label{ABP-fullversion}
	Consider the surface $\Gamma = B_1 \cap \{x_n = \psi(x')\}$ and  $u$ functions satisfying the conditions
	\begin{equation*}
		\left\{
		\begin{array}{rclcl}
			u\in S(\lambda,\Lambda,f^{\pm}) &\mbox{in}&   \Omega^{\pm} \\
			u^+_{\nu} - u^{-}_{\nu} = g(x)  &\mbox{on}&  \Gamma,
		\end{array}
		\right.
	\end{equation*}
	with $f^{\pm} \in C(\Omega^{\pm}) \cap L^{\infty}(B_1)$, $g \in L^{\infty}(\Gamma)$, and $\psi \in C^{1,\alpha}(\overline{B'_1})$, for some $0 < \alpha < 1$. Then the following estimate is true 
		\begin{eqnarray*}
		\|u\|_{L^{\infty}(\Omega)}\leq \|u\|_{L^{\infty}(\partial B_1)}+C(\Vert g\Vert_{L^{\infty}(\Gamma)}+\Vert f^{+}\Vert_{L^{n}(\Omega^+)} + \Vert f^{-}\Vert_{L^{n}(\Omega^-)})
	\end{eqnarray*}
	where $C$ is a constant which depends   only on  the quantities $n$, $\lambda$, $\Lambda$.
\end{lemma}

\subsection{Uniqueness and Existence of Viscosity Solutions of flat interface problems} This subsection is concerned with the Uniqueness and Existence of viscosity solution to fully nonlinear elliptic transmission problems
\begin{equation} \label{Resul}
\left\{
\begin{array}{rclcl}
	F(D^2u) &=& f^{\pm} & \mbox{in} & B^{\pm}_1,\\
	u^+_{\nu} - u^{-}_{\nu}&=& g & \mbox{on} & T = B_1 \cap \{x_n=0\}\\
	u&=&\varphi & \mbox{on} & \partial B_1
\end{array}
\right.
\end{equation}
The next proof follows the ideas from \cite[Theorem 4.7]{MPA} with
minor modifications. For this reason, we will omit it here.
\begin{theorem}\label{comparação}
	Let $f^{\pm}_1, f^{\pm}_2 \in C(B^{\pm}_1) \cap L^{\infty}(B^{\pm})$ and $g_1,g_2 \in C(T)$.  Also, let  $u \in \textrm{USC}(\overline{B_1}), v \in \textrm{LSC}(\overline{B_1})$ be bounded functions satisfying 
	$$
	\left\{
	\begin{array}{rclcl}
		F^{\pm}(D^2u) &\geq& f_{1} & \mbox{in} & B^{\pm}_1,\\
		u^{+}_{x_n} - u^{-}_{x_n}&\geq& g_{1} & \mbox{on} & T\\
	\end{array}
	\right.
	\quad
	\textrm{and}
	\quad
	\left\{
	\begin{array}{rclcl}
		F^{\pm}(D^2v) &\leq& f_{2} & \mbox{in} & B^{\pm}_1,\\
		v^+_{x_n} - v^{-}_{x_n}&\leq& g_{2} & \mbox{on} & T\\
	\end{array}
	\right.
	$$
	in the viscosity sense. It holds that
	$$
	\left\{
	\begin{array}{rclcl}
		w \in \underline{S}\left(\frac{\lambda}{n},\Lambda,f^{\pm}_{1}-f^{\pm}_{2}\right) & \mbox{in} & \Omega,\\
		w^+_{x_n} -w^{-}_{x_n} \geq g_{1}-g_{2} & \mbox{on} & T,\\
	\end{array}
	\right.
	$$
where  $w=u-v$ ,	in the viscosity sense.
\end{theorem}

\begin{theorem}[{\bf Existence and Uniqueness}]\label{Unicidade}
	Suppose that the conditions  $(A1),$ and $(SC)$ hold. Consider functions $f^{\pm} \in C(B^{\pm} \cup T) \cap L^{\infty}(B_1)$, $g \in C(T)$, and $\varphi \in C(\partial B_1)$. Then there exists a unique viscosity solution $u \in C(\overline{B_1})$ of \eqref{Resul}.	\end{theorem}
\begin{proof}
The result follows by the direct application of \cite[Corollary 4.8]{MPA} and \cite[Theorem 4.11]{MPA} in conjunction with the ABP estimate given in Lemma \ref{ABP-fullversion}.
\end{proof}


\subsection{Main assumptions}
In this subsection, we detail the main assumptions used throughout the paper.  

\vskip .5cm  
\begin{enumerate}
\item[{\bf A1.}] {\bf [Reducibility condition]}  We will assume that $F(0)=0$ and $f^{+}(0)=f^{-}(0)=0$. 
\end{enumerate}
 
\vskip .2cm 
 
\begin{enumerate}
\item[{\bf A2.}] {\bf [Interface Regularity]} The interface $\Gamma = B_1 \cap \{x_n = \psi(x')\}$ is given by the graph of a function $\psi \in C^2(\mathbb{R}^{n-1})$. Moreover, we suppose that $|g(0)| \cdot \|D^2_{x'} \psi(0)\| =0$.
\end{enumerate}

\vskip .2cm 

\begin{enumerate}
\item[{\bf A3.}] {\bf [Integrabilility of the source term]} The functions $f^{\pm}$ are continuous at $0$ and $f^{\pm} \in C^{0,\alpha}(0)$ with
$$
\left( \intav{B_r \cap \Omega^{\pm}} |f^{\pm}(x) |^n dx \right)^{1/n} \le \mathfrak{c}_{f^{\pm}} \cdot r^{\alpha},
$$
for all $r>0$ small. 
\newline
\item[{\bf A4.}] {\bf [Transmission boundary condition.]} We suppose that $g \in C^{0,\alpha}(0)$. 
\end{enumerate}
\vspace{0.5 cm}

\begin{remark} A few remarks are in order:
\begin{itemize}
\item[(i)]  It is interesting to note that our results indeed provided $C^{2,\alpha}$ at the origem, i.e., $C^{2,\alpha}(0)$. It is not difficult to verify that $C^{2,\alpha}(0)$ implies $C^{2,\alpha}(x_0)$. 
\item[(ii)] Notice that {\bf A1.}, is not restrictive.   In fact, otherwise, by unifor ellipticity, there exists $s \in \mathbb{R}$ such that $F(s Id_{n \times n})= f^{\pm}(0)$ and $|s| \le \frac{|f^{\pm}(0)|}{\lambda}$. Hence, consider $\tilde{F}(M)= F(M+s Id_{n \times n})-f^{\pm}(0)$ and $v = u - \frac{s}{2}|x|^2$. Then $\tilde{F} \in \mathcal{E}$$(\lambda,\Lambda)$, $\tilde{F}(0)=0$, and $\tilde{F}(D^2 v)= f^{\pm}-f^{\pm}(0)$. 
\item[(iii)] The hypothesis {\bf A2.}, which plays an important role in reference \cite{MPA}, allowed the authors to derive certain global $C^{2,\alpha}$ estimates for the problem considered in their work that includes the case of a curved surface. Such a requirement will be necessary due to the fact that some arguments from \cite{MPA}, which depend on the aforementioned hypothesis, will be applied in this manuscript.
\end{itemize}
\end{remark}

 Our first result of this paper gives a theoretical contribution to such a question. Namely, the first result of this paper is
 \vspace{0.5 cm}


\begin{theorem}[{\bf $C^{2,\alpha}$ estimate under Small Ellipticity Aperture}]\label{BMO}
Let $\alpha\in (0,1)$ be arbitrary. Assume the hypotheses {\bf A1.}- {\bf A4.} hold and  that $0$ lies within $\Gamma$. Then there exists $\delta>0$ such that if 
  \begin{equation} \label{Cor}
  \frac{\Lambda}{\lambda}-1 \le \delta,
  \end{equation}
    then any bounded viscosity solution $u$ to the problem \eqref{E1} satisfies $u^{\pm} \in C^{2,\alpha}(0)$, that is, there are   polynomials with quadratic growth given by 
  $$
  \mathfrak{P}^{\pm}(x) = \frac{1}{2} x^t \cdot \mathcal{A}^{\pm} \cdot x + \mathcal{B}^{\pm} \cdot x + \mathfrak{c}
  $$
 satisfying 
  $$
  	\|u^{\pm} - \mathfrak{P}^{\pm}\|_{L^{\infty}(\Omega^{\pm}_r)} \le C r^{2+\alpha}, \,\,\, \textrm{for all} \,\,\, r \ll 1
  $$
  with  $C_0 >0$ depending only on $n,\lambda, \Lambda$ and $\alpha$. Moreover, the estimate below holds true
	$$
		\|\mathcal{A}^{\pm}\|_{\textrm{Sym}(n)} + |\mathcal{B}^{\pm}| + |\mathfrak{c}| \le C_0  \|\psi\|_{C^{2,\alpha}(0)}\left(  \|g\|_{C^{1,\alpha}(0)} + [f^+]_{C^{\alpha}(0)} + [f^-]_{C^{\alpha}(0)} \right).
	$$
\end{theorem}

Theorem \ref{BMO} improves the result established in \cite{MPA} because we prove the result for all $\alpha \in (0,1)$. The idea of the proof of  Theorem \ref{BMO}  is based  in \cite{MPA}  along with a compactness approach as in \cite{Caff1}, see also \cite{JV1}.The insights however come from the prominent heuristics explained in \cite{JV1}. We shall interpret the Laplace equation as the {\it geometric tangential equation} of the {\it manifold} formed by fully nonlinear elliptic operators $F$ as $\mathfrak{e} := 1- \frac{\lambda}{\Lambda} \rightarrow 0$. In turns, we show that at every scale it is possible to find a harmonic function close to a viscosity solution to \eqref{E1}. The degree of closeness will be measured by how close the ellipticity constants of $F$ are of being constant, i.e., provided the ellipticity aperture is small enough. Iterating such an argument in $\rho-$adic balls shows that the graph of a viscosity solution $u$ can be approximated in $B^{\pm}_\rho$ by a quadratic polynomial with an error of order $\sim \mbox{O}\left(\rho^{2+\alpha} \right)$.

Our second result addresses a Schauder estimate for \eqref{E1} when the operator is assumed only convexity of the superlevel (sublevel) sets, which means that $F$ is quasiconcave (resp. quasiconvex), or,  alternatively, we impose some asymptotic concavity properties, namely $F=F(M)$ is concave (resp. convex or ``close to a linear function") when $M$ is large.  For instance, by way of illustration, such a limiting profile $F^{\star}$  appears naturally in singularly perturbed free boundary problems ruled by fully nonlinear equations, whose Hessian of solutions
blows-up through the phase transition of the model, i.e., $\partial\{u^{\varepsilon}> \varepsilon\}$, where $u^{\varepsilon}$ satisfies in the viscosity sense
$$
F(D^2 u^{\varepsilon}) = \mathcal{Q}_0(x)\frac{1}{\varepsilon} \zeta \left(\frac{u^{\varepsilon}}{\varepsilon}\right).
$$
For such approximations, we have $0< \mathcal{Q}_0 \in C^0(\overline{\Omega})$, $0\leq \zeta \in C^{\infty}(\R)$  with $\supp ~\zeta = [0,1]$. For this reason, in the above model, the limiting free boundary condition is governed by the F$^{\star}$ rather than F, i.e.,
$$
F^{\star}( \nabla u(z_0) \otimes \nabla u(z_0)) = 2\mathrm{T},  \quad z_0 \in \partial\{u_0>0\}
$$
in some appropriate viscosity sense, for a certain total mass $\mathrm{T}>0$ (cf. \cite[Section 6]{RT} for some enlightening example and details). Furthermore, limit profiles  also appears in \cite{Al} for an account of the Schauder theory of viscosity solutions of $F(D^2 v)=0$ in $B_1$ under weak concavity assumptions. 

\begin{definition}
 For a uniformly elliptic operator $F : \textrm{Sym}(n) \rightarrow \mathbb{R}$, we say that it is quasiconcave (resp. quasiconvex) if it is a quasiconcave (quasiconvex) function of $M \in \textrm{Sym}(n)$, namely for all $M_1, M_2 \in \textrm{Sym}(n)$ and $\theta \in [0,1]$ we have
 $$
 F((1-\theta) M_1 + \theta M_2) \ge \min \{F(M_1), F(M_2)\} \, (\textrm{respc.} \,\, \le \max\{F(M_1),F(M_2)\}).
 $$
 \end{definition}

Useful to the subsequent analysis, we define 
\begin{equation} \label{Hom1}
\alpha_{0} \colon= \sup \left\{ \beta \in (0,1) \,\, \Biggl |
\begin{array}{ccccccc}
\exists \,\, \mathfrak{C}_{\beta} >0 \,\, \textrm{such that} \,\, \|v\|_{C^{2,\beta}(B_{1/2})} \le \mathfrak{C}_{\beta} \|v\|_{L^{\infty}(B_1)}  \\
 \textrm{for any viscosity solution of} \,\, F(D^2 v)=0  \\ \textrm{in} \,\,  B_1, \textrm{where}  F \in \mathcal{E}(\lambda,\Lambda) \,\, \textrm{is a}\\ \,\,
 \textrm{quasiconvex/quasiconcave operator}
 \end{array}
\right\}.
\end{equation}

The existence of $\alpha_0 \in (0,1)$ is guaranteed by the following result. 

\begin{proposition}[\cite{Al}, Theorem 4.1]\label{usar}
Let $F \in \mathcal{E}(\lambda,\Lambda)$ be quasiconvex and continuous, $\varphi \in C(\partial B_1)$  and $u$ be a continuous viscosity solution to 
$$
\left\{
	\begin{array}{rclcl}
		F(D^2 u) &=& 0 & \mbox{in} & B_1,\\
		u&=& \varphi & \mbox{on} & \partial B_1
	\end{array}
	\right.
$$
Then $u \in C^{2,\alpha_0}_{loc}(B_1)$ for some $\alpha_0 \in (0,1)$ with the estimate below  holding true
$$
\|u\|_{C^{2,\alpha_0}(\overline{B_{1/2}})} \le C \|u\|_{L^{\infty}(B_1)},
$$
where $C$ is a constant that depends only on $ n,\lambda,$ and $\Lambda$.
\end{proposition}

Finally, we can state our second result.

\begin{theorem}[{\bf $C^{2,\alpha}$ estimates for viscosity solutions of quasiconcave equations}]\label{BMOQ}
Let $F \in \mathcal{E}(\lambda,\Lambda)$ be a quasiconvex or quasiconcave operator. Fix any $\alpha \in (0,\alpha_0)$ and assume the hypotheses  {\bf A1.--A4.}. If  $u$ is a bounded viscosity solution to the problem \eqref{E1}, then $u^{\pm} \in C^{2,\alpha}(0)$, in the sense that  there are  quadratic polynomials 
  $$
  \mathfrak{P}^{\pm}(x) = \frac{1}{2} x^t \cdot \mathcal{A}^{\pm} \cdot x + \mathcal{B}^{\pm} \cdot x + \mathfrak{c},
  $$
  such that
  $$
  	\|u^{\pm} - \mathfrak{P}^{\pm}\|_{L^{\infty}(\Omega^{\pm}_r)} \le C r^{2+\alpha}, \,\,\, \textrm{for all} \,\,\, r \ll 1
  $$
  with  $C_0 >0$ depending only on $n,\lambda, \Lambda$ and $\alpha$. Moreover, the estimate below holds true
	$$
		\|A^{\pm}\|_{\textrm{Sym}(n)} + |B^{\pm}| + |c| \le C_0  \|\psi\|_{C^{2,\alpha}(0)}\left(  \|g\|_{C^{1,\alpha}(0)}+[f^+]_{C^{\alpha}(0)} + [f^-]_{C^{\alpha}(0)} \right).
	$$
\end{theorem}

To develop our theory, we will utilize the ideas contained in \cite{MPA} and employ the so-called geometric tangential mechanisms. By doing so, we can obtain good regularity estimates (cf. \cite{WN}) and, in a suitable manner, transfer such estimates to solutions of the original problem \eqref{E1} via compactness and stability processes.

Finally, a few comments. It is interesting to note that Theorem \ref{BMO} (or Theorem \ref{BMOQ}) provides pointwise regularity of $u$ based on the respective pointwise regularity of $f^{\pm}$. It is not difficult to verify that pointwise estimates imply global estimates, as stated in the Corollary \ref{Obs1}.


\section{Schauder type Estimate under Small Ellipticity Aperture}
 In this section  we derive Schauder estimates for viscosity solutions to non-convex fully nonlinear elliptic transmission problems.  
 We assume the Cordes-Landis type conditions on the ellipticity conditions, i.e, when $\frac{\lambda}{\Lambda}\le1 + \delta$.
 
  We begin with some preliminary results.
  
  \begin{proposition}[Interior derivative estimates] \label{Grad}
  
  Let $0 < r \le 1$. Suppose that $\mathfrak{h}$ is a bounded viscosity solution of the transmission problem
   \begin{equation} \label{flat1}
	\left\{
	\begin{array}{rclcl}
		\Delta \mathfrak{h} &=& 0 & \mbox{in} & B^{\pm}_1,\\
		\mathfrak{h}^{+}_{x_n} - \mathfrak{h}^{-}_{x_n}&=& 0 & \mbox{on} & T.
	\end{array}
	\right.
	\end{equation}
	Then, for any radii $0 < \rho \le \frac{r}{2}$, it holds that $\mathfrak{h} \in C^{3}(\overline{B^{\pm}_{\rho}})$ with the following estimates holding true
	\begin{eqnarray*}
	\displaystyle\osc_{B^{\pm}_{\rho}} \left( \mathfrak{h}^{\pm} - \frac{1}{2} x^t \cdot D^2 \mathfrak{h}^{\pm}(0) \cdot x - \nabla \mathfrak{h}(0) \cdot x\right)& \le& C \left( \frac{\rho}{r}\right)^{3} \cdot \osc_{\overline{B_r}} \, \mathfrak{h}\\
	r^2 \|D^2 \mathfrak{h}^{\pm}(0)\|_{\textrm{Sym}(n)} + r |\nabla \mathfrak{h}(0) |&\le& C \cdot \osc_{\overline{B_r}} \, \mathfrak{h},
	\end{eqnarray*}
	where $C>0$ is a constant that depends only on $n,\lambda,$ and $ \Lambda$.
\end{proposition}
\begin{proof}
In fact, fix $0 < \rho \le \frac{r}{2}$. Let $w$ be the unique viscosity solution of the Dirichlet problem
$$
	\left\{
	\begin{array}{rclcl}
		\Delta w &=& 0 & \mbox{in} & B^{\pm}_1,\\
		w&=& \mathfrak{h} & \mbox{on} & T.
	\end{array}
	\right.
	$$
	We know that $w \in C^{3}_{\textrm{loc}}(B_1)$ and $\|w\|_{C^{3}(\overline{B_{1/2}})} \le C \|w\|_{L^{\infty}(B_1)}$. By uniqueness of viscosity solution to flat interface problems (see Theorem \ref{Unicidade} ) it follows that $w=\mathfrak{h}$ in $B_1$.  Applying the above result to the function  $v(x)=\mathfrak{h}(rx)-\mathfrak{h}(0)$, for $x \in B_1$, we get $\| v \|_{C^3(\overline{B_{1/2}})} \le C \|v\|_{L^{\infty}(B_1)}$.  Thus, it follows from  the Mean Value Theorem that 
	\begin{eqnarray*}
	\osc_{B^{\pm}_{\rho}} \left( \mathfrak{h}^{\pm} - \frac{1}{2} x^t \cdot D^2 \mathfrak{h}^{\pm}(0) \cdot x - \nabla \mathfrak{h}(0) \cdot x\right) &\le& C \rho^{3}  \|D^3 \mathfrak{h}^{\pm}\|_{L^{\infty}(\overline{B^{\pm}_{\rho}})} \\
	&\le& C \left(\frac{\rho}{r} \right)^3 \|v\|_{L^{\infty}(B_1)} \\
	&\le& C  \left(\frac{\rho}{r} \right)^3 \osc_{\overline{B_r}} \mathfrak{h}.
	\end{eqnarray*}
	Moreover,
	$$
	r^2 \|D^2 \mathfrak{h}^{\pm}(0)\|_{\textrm{Sym}(n)} + r |\nabla \mathfrak{h}(0)| \le C \|v\|_{L^{\infty}(B_1)} \le  C \osc_{\overline{B_r}} \mathfrak{h}.
	$$
\end{proof}

\subsection{Compactness Result } \label{A1}

In this section, we present the proof of Theorem \ref{BMO}. An important step towards the proof consists of some key results that will play a crucial role in our strategy.


\begin{lemma}\label{Lema1}
Suppose that $\Gamma= B_1 \cap \{x_n = \psi(x')\}$, for some function $\psi \in C^2(B'_1)$. Fix $D' \in \mathbb{R}^{n-1}$ and $b \in \mathbb{R}$.  Then, given any $\epsilon>0$ there exists $\delta >0$ (depending only on $\epsilon, n, \lambda,\Lambda$) such that if
$$
	\max \{\|\psi\|_{C^2(B'_1)},  \|g- D' \cdot x -b\|_{L^{\infty}(\Gamma)} , \|f^{\pm}\|_{L^{\infty}(\Omega^{\pm})} \} \le \delta \quad \textrm{and} \quad \Lambda \le (1+\delta) \lambda,
$$
then viscosity solutions $u$ and $\mathfrak{h}$ of
$$
	\left\{
	\begin{array}{rclcl}
		F(D^2 u) &=& f^{\pm} & \mbox{in} & \Omega^{\pm},\\
		u^{+}_{\nu} - u^{-}_{\nu}&=& g & \mbox{on} & \Gamma,
	\end{array}
	\right.
	$$
 and 
 \begin{equation}\label{Laplace}
	\left\{
	\begin{array}{rclcl}
		\Delta \mathfrak{h} &=& 0 & \mbox{in} & B^{\pm}_{3/4},\\
		\mathfrak{h}^{+}_{x_n} - \mathfrak{h}^{-}_{x_n}&=& D'\cdot x'+b & \mbox{on} & B_{3/4} \cap \{x_n=0\},
	\end{array}
	\right.
	\end{equation}
respectively, satisfy
$$
\|u - \mathfrak{h}\|_{L^{\infty}(B_{3/4})} \le \epsilon.
$$
\end{lemma}
 \begin{proof}
 We will proceed by a \textit{reductio at absurdum} argument. Thus, let us assume that the thesis is false. Then, there exist $\epsilon_0>0$ and sequences of functions $(F_j)_{j \in \mathbb{N}}$, $(u_{j})_{j \in \mathbb{N}}$,  $(f^{\pm}_j)_{j \in \mathbb{N}}$, $\{g_j\}_{j \in \mathbb{N}}$, $(\psi_{j})_{j\in\mathbb{N}}$ and $(\Lambda_j)_{j\in\mathbb{N}}$ satisfying 
	$$
	\left\{
	\begin{array}{rclcl}
		F_j(D^2u_j) &=& f^{\pm}_j(x)  & \mbox{in} &  \Omega^{\pm}_j,\\
		(u^+_{j})_{\nu} - (u^{-}_j)_{\nu} &=& g_j & \mbox{on} & \Gamma_j = B_1 \cap \{x_n = \psi_j(x)\}
	\end{array}
	\right.
	$$
	with $\|u_k\|_{L^{\infty}(B_1)}\le 1$, where
	\begin{equation} \label{S}
\|f^{\pm}_j\|_{L^{\infty}(\Omega^{\pm}_j)}+ \|g_{j}-D'\cdot x' - b\|_{L^{\infty}(\Gamma_j)}+\|\psi_j\|_{C^2(B'_1)} \le \frac{1}{j} \,\,\, \textrm{and} \,\,\,\, \Lambda_j \le \left(1+\frac{1}{j} \right) \lambda
\end{equation}
 and such that
	\begin{equation} \label{1L}
		\|u_j- \mathfrak{h}\|_{L^{\infty}(B_{\frac{3}{4}})} > \epsilon_0
	\end{equation}
for any viscosity solution $\mathfrak{h}$ of \eqref{Laplace}.  From global elliptic regularity theory (cf. \cite[Theorem 1.1]{MPA}) we have 
	\begin{equation} \label{4.6}
		\|u_j\|_{C^{\alpha_1}(\overline{B_{3/4}}  )}\le C(n,\lambda,\Lambda,\mathfrak{c}_{1})
	\end{equation}
for some $\alpha_1 = \alpha_1(n,\lambda,\Lambda) \in (0,1)$ and sufficiently large $j$.  The sequence $\{u_{j}\}_{j \ge 1}$ is therefore bounded and equicontinuous on $\overline{B_{3/4}}$. By compactness, up to a subsequence, $u_j \to u_{\infty}$ uniformly on compact subsets of $B_{3/4}$ as $j \to + \infty$. 
	
On the other hand,  from \eqref{S} we have that $\Lambda_j \le 2 \lambda$ for all $j \gg 1$, which implies that the operators $F_j \in \mathcal{E}$$(\lambda, 2 \lambda)$ for $j \gg 1$. Moreover, by using \eqref{S}, we can conclude that $\Lambda_j \to \lambda$, up to a subsequence.  Hence, $F_{j} \to \lambda \textrm{Tr}$ locally uniformly on $\textrm{Sym}(n)$ and by Stability Lemma (Lemma \ref{Est}) $u_{\infty}$ satisfies
	\begin{equation} \label{Unic}
		\left\{
		\begin{array}{rclcl}
			\Delta u_{\infty} &=& 0  & \mbox{in} &  B^{\pm}_{3/4}, \\
			(u^{+}_{\infty})_{x_n} - (u^{-}_{\infty})_{x_n} &=& D' \cdot x' + b  & \mbox{on} & T_{3/4}
		\end{array}
		\right.
	\end{equation}
in viscosity sense.  This contradicts \eqref{1L}.

	 \end{proof}


As in \cite{MPA}, in the previous lemma we assumed that $u \in C(B_1)$. In later proofs, we will also need a similar approximation result when $u$ has a jump discontinuity across the interface. We prove this version in the next lemma. 

\begin{lemma}\label{Lema2}
Assume that $\Gamma= B_1 \cap \{x_n = \psi(x')\}$ for some $\psi \in C^2(B'_1)$, and that $f^{\pm}$ satisfy
$$
\left( \intav{B_r(x_0) \cap \Omega^{\pm}} |f^{\pm}|^n dx \right)^{1/n} \le C_{f^{\pm}} \cdot r^{\alpha-1},
$$
for all small radii $r>0$, $x_0 \in \Omega^{\pm} \cup \Gamma$, and some $\alpha \in (0,1)$. Fix $D' \in \mathbb{R}^{n-1}, b \in \mathbb{R}$. If $u\in C(B_1 \setminus \Gamma) \cap L^{\infty}(B_1)$ a viscosity solution of the problem
$$
	\left\{
	\begin{array}{rclcl}
		F(D^2 u) &=& f^{\pm} & \mbox{in} & \Omega^{\pm},\\
		u^{+}_{\nu} - u^{-}_{\nu}&=& g & \mbox{on} & \Gamma
	\end{array}
	\right.
	$$
with $\|u\|_{L^{\infty}(B_1)} \le 1$, $u^+ -u^{-} \equiv h \in C^2(\Gamma)$, then  given an arbitrary  $\epsilon>0$, there is $\tilde{\delta}>0$ such that if 
\begin{equation} \label{Estima}
	\max \{\|\psi\|_{C^2(B'_1)}, \|g- D' \cdot x -b\|_{L^{\infty}(\Gamma)} , \|f^{\pm}\|_{L^{\infty}(\Omega^{\pm})} \} \le \tilde{\delta} \quad \textrm{and} \quad \Lambda \le (1+\tilde{\delta}) \lambda,
\end{equation}
then there exists a bounded viscosity solution  $\mathfrak{h} \in C(B_{1/2})$ to
 \begin{equation}\label{Laplace1}
	\left\{
	\begin{array}{rclcl}
		\Delta \mathfrak{h} &=& 0 & \mbox{in} & B^{\pm}_{1/2},\\
		\mathfrak{h}^{+}_{x_n} - \mathfrak{h}^{-}_{x_n}&=& D'\cdot x'+b & \mbox{on} & B_{3/4} \cap \{x_n=0\}
	\end{array}
	\right.
	\end{equation}
such that
$$
\|u-\mathfrak{h}\|_{L^{\infty}(B_{1/2})} \le \epsilon.
$$

\end{lemma}

\begin{proof}
  Using Lemma \ref{Lema1}, the proof follows exactly as the one of Lemma 5.8 in \cite{MPA}.
 \end{proof}

\subsection{A quadratic polynomials approximation  }

In this subsection, we present the core oscillation decay that will ultimately imply $C^{2,\alpha}(0)$-regularity for solutions to the problem \eqref{E1}. The first task is a step-one discrete version of the aimed regularity estimate. This is the content of the next lemma. As in \cite{MPA}, here we also require that $\textbf{A2.}$ be satisfied.

\begin{lemma}[A quadratic approximation]\label{LSS2}
Given $\alpha \in (0,1)$, there are  constants $\mathfrak{C}_0 >0,$ and $0 < \epsilon , \rho < 1/2$, depending only on $n,\lambda,\Lambda$ and $\alpha$, such that, if $u$ satisfies the assumptions of Lemma \ref{Lema2} with $D'= \nabla'g(0)$ and $b=g(0)$, and $\|g\|_{C^1(\Gamma)} \le 1$, then there are  polynomials with quadratic growth
$$
\mathfrak{P}^{\pm}(x) = \frac{1}{2} x^t \cdot \mathcal{A}^{\pm} \cdot x + \mathcal{B}^{\pm} \cdot x + \mathfrak{c},
$$
with $\mathcal{A}^{\pm} \in \textrm{Sym}(n)$, $\mathfrak{B}^{\pm} \in \mathbb{R}^n$  and $\mathfrak{c} \in \mathbb{R}$ satisfying $\textrm{Tr}(\mathcal{A}^{\pm})=0$,
\begin{eqnarray}
	\mathcal{B}^+_i - \mathcal{B}^{-}_i = 0 \,\,\, \textrm{if} \,\,\, i < n \quad &\textrm{and}& \quad \mathcal{B}^+_n - \mathcal{B}^{-}_n = g(0), \label{S3}\\
\mathcal{A}^{+}_{ij} - \mathcal{A}^{-}_{ij} =0 \,\,\, \textrm{if} \,\,\, i,j < n \,\,\, &\textrm{and}& \,\,\, \mathcal{A}^{+}_{nj}-\mathcal{A}^{-}_{nj} = g_{x_j}(0) \,\,\, \textrm{if} \,\,\, j < n,\label{S4}
\end{eqnarray}
and
$$
\|\mathcal{A}^{\pm}\|_{\textrm{Sym}(n)} + |\mathcal{B}^{\pm}| + |\mathfrak{c}| \le \mathfrak{C}_0 ,
$$
such that
$$
\|u^{\pm} - \mathfrak{P}^{\pm}\|_{L^{\infty}(\Omega^{\pm}_{\rho})} \le \rho^{2+\alpha}.
$$
\end{lemma}
 \begin{proof} The proof follows the same lines as that of \cite[Lemma 7.1]{MPA} along with minor changes. In fact, fix $0 < \epsilon, \rho < 1/2$ to be determined. By Lemma \ref{Lema2}, there exists a bounded viscosity solution $v \in C(B_{1/2})$ to
\begin{equation} \label{Estima3}
\left\{
\begin{array}{rclcl}
\Delta v &=& 0 & \mbox{in} & B^{\pm}_{1/2},\\
v^{+}_{x_n} - v^{-}_{x_n}&=& \nabla' g(0)\cdot x'+g(0) & \mbox{on} & B_{3/4} \cap \{x_n=0\},
\end{array}
\right.
\end{equation}   
such that
\begin{equation} \label{Estima4}
\|u-v\|_{L^{\infty}(B_{1/2})} \le \epsilon.
\end{equation}
Now, since $\nabla'g(0) \cdot x' + g(0)$ is smooth and the interface is flat, from Lemma \ref{Grad}, we have  
$$
\|v^{\pm}\|_{C^{3}(\overline{B^{\pm}_{1/3}})} \le C_0,
$$
for some constant $C_0 = C_0(n,\lambda, \Lambda)>0$. Let 
   $$
   \mathfrak{P}^{\pm}(x) \colon= \frac{1}{2} x^t \cdot D^2 v^{\pm}(0) \cdot x + \nabla v^{\pm}(0) \cdot x + v(0).
   $$
   By the  previous estimate, it follows that $\|D^2 v^{\pm}(0)\|_{\textrm{Sym}(n)} + |\nabla v^{\pm}(0)| + |v(0)| \le C_0$ and, for $\rho \in (0,1/3)$,
   \begin{equation} \label{Estima5}
   \|v^{\pm} -\mathfrak{P}^{\pm}\|_{L^{\infty}(B^{\pm}_{\rho})} \le C_0 \cdot \rho^3.
   \end{equation}
   Now, choose $\rho$ small enough such that $C_0 \rho^{3} \le \frac{1}{2} \rho^{2+\alpha}$.  Then choose $\epsilon = \frac{1}{2} \rho^{2+\alpha}$.  Combining \eqref{Estima4} and \eqref{Estima5}, we get
   \begin{eqnarray*}
   \|u^{\pm} - \mathfrak{P}^{\pm}\|_{L^{\infty}(\Omega^{\pm}_{\rho})} &\le&   \|u - v\|_{L^{\infty}(B_{1/2})} +   \|v^{\pm} - \mathfrak{P}^{\pm}\|_{L^{\infty}(\Omega^{\pm}_{\rho})}  \le \epsilon + C_0 \rho^3 \le \rho^{2+\alpha}.
   \end{eqnarray*}
   Moreover, since $v^{+}=v^{-}$ on $\{x_n=0\} \cap B_{1/2}$ and $v$ is $C^2$ up to the flat interface, we see that $\nabla'v^+(0)=\nabla'v^{-}(0)$ and $D^2_{x'} v^{+}(0)=D^2_{x'} v^{-}(0)$. In particular, $\mathcal{B}^+_{i} - \mathcal{B}^{-}_i =0$ and $\mathcal{A}^{+}_{ij} - \mathcal{A}^{-}_{ij} =0$ if $i,j < n$. From the transmission condition in \eqref{Unic}, we get $\mathcal{B}^{+}_{n} - \mathcal{B}^{-}_n = v^+_{x_n} - v^{-}_{x_n} = g(0)$ and $\mathcal{A}^{+}_{nj} - \mathcal{A}^{-}_{nj} = v^{+}_{x_n x_j}(0) - v^{-}_{x_n x_j}(0) = g_{x_j}(0)$ for all $j < n$.  	
\end{proof}

In the sequel, we shall iterate Lemma \ref{LSS2} in appropriate dyadic balls in order to obtain the precise sharp oscillation decay of the difference between $u$ and the quadratic polynomials $\mathfrak{P}^{\pm}_k$.  

\section{Proof of Theorem \ref{BMO}}

The proof of the result will be divided into several stages. 
\subsection{Step 1: Smallness regime}

Initially,  we comment on the scaling features of the transmission problems that allow us to reduce the proof of Theorem \ref{BMO} to the hypotheses of Lemma \ref{Lema2} and Lemma \ref{LSS2}. 

Let $\mathfrak{C}_0, \rho, \epsilon>0$ be the  constants given in Lemma \ref{LSS2} and $\delta>0$ be the one given in Lemma \ref{Lema1}.  Before delving into the details of this initial step, fix the number $\delta_0 > 0$ specified in the proof of Lemma \ref{L02}. Since this $\delta_0$ does not depend on the arguments of this first step, we've chosen to fix $\delta_0$ at this moment to improve the  readability of the manuscript. We can suppose that
\begin{enumerate}
\item[(i)] $\psi(0')=0, \, \nabla' \psi(0')=0'$; $\|D^2_{x'} \psi(0')\| \le \delta_0$, and $[\psi]_{C^{2,\alpha}(0)} \le \delta_0 ,$
\item[(ii)] $\|u\|_{L^{\infty}(B_1)} \le 1$ and $\|g\|_{C^1(\Gamma)} \le 1 ,$
\item[(iii)] $[g]_{C^{1,\alpha}(0)} + \mathfrak{c}_{f^-} + \mathfrak{c}_{f^+} \le \delta_0 ,$
\end{enumerate}
replacing if necessary $u$ by
\begin{equation} \label{Norm}
	w(y) \colon= \frac{u(\eta y)}{\tau}
\end{equation}
for parameters $\eta$ and $\tau$ to be determined.  In fact, let $u \in C(B_1)$ be a viscosity solution to 
$$
\left\{
\begin{array}{rclcl}
	F(D^2u) &=& f^{\pm} & \mbox{in} & \Omega^{\pm},\\
	u^+_{\nu} - u^{-}_{\nu}&=& g & \mbox{on} & \Gamma
\end{array}
\right.
$$
Define $w : B_1 \rightarrow \mathbb{R}$ as in \eqref{Norm}.  We readily check that $w$ solves
\begin{equation} \label{Peq_1}
\left\{
\begin{array}{rclcl}
	F_{\eta,\tau}(D^2w) &=& f^{\pm}_{\eta,\tau} & \mbox{in} & \tilde{\Omega}^{\pm},\\
	w^+_{\nu} - w^{-}_{\nu}&=& g_{\eta,\tau} & \mbox{on} &\tilde{ \Gamma}
\end{array}
\right.
\end{equation}
where 
$$
\tilde{\Omega}^{\pm} \colon= \{y \in \tilde{\Omega} :  \eta y \in \Omega^{\pm}\} , \,\, \tilde{\Gamma} \colon= \{y \in B_1 :  y_n=\psi_{\eta}(y) \},
$$
with $\psi_{\eta}(y) = \eta^{-1} \psi(\eta y)$ and
$$
F_{\eta,\tau}(M) \colon= \frac{\eta^2}{\tau} F \left( \frac{\tau}{\eta^2} M\right), \,\,\, f^{\pm}_{\eta, \tau}(y) \colon= \frac{\eta^2}{\tau} f^{\pm}(\eta y) \quad \textrm{and} \quad g_{\eta,\tau}(y) \colon= \frac{\eta}{\tau} g(  \eta y).
$$
Easily one verifies that $F_{\eta,\tau} \in \mathcal{E}$$(\lambda, \Lambda)$.  Thus, it satisfies the same Cordes condition as $F$. Moreover,   $f^{\pm}_{\eta,\tau}$ satisfy
\begin{eqnarray} 
\left( \intav{B_r \cap \tilde{\Omega}^{\pm}} |f^{\pm}_{\eta,\tau}(y)|^n dy\right)^{1/n} &=& \frac{\eta^2}{\tau} \left( \intav{B_{\eta r} \cap \Omega^{\pm}} |f^{\pm}(x)-f^{\pm}(0)|^n \right)^{1/n} \nonumber \\ &\le& \frac{\eta^{2}}{\tau} \cdot \mathfrak{c}_{f^{\pm}} (r \eta)^{\alpha} = \frac{\eta^{2+\alpha}}{\tau} \mathfrak{c}_{f^{\pm}} \cdot r^{\alpha} \le   \tilde{\mathfrak{c}}_{f^{\pm}} \cdot r^{\alpha-1}.\label{Peq0}
\end{eqnarray}
where $\tilde{\mathfrak{c}}_{f^{\pm}_{\eta,\tau}} = \frac{\eta^{2+\alpha}}{\tau} \mathfrak{c}_{f^{\pm}}$. 

\begin{itemize}
\item[(i)] If $y \in \tilde{\Gamma}$, then $y_n = \tilde{\psi}(y')$, with $\tilde{\psi}(y') = \frac{1}{\eta} \psi(\eta y')$. Thus, choosing $$\eta \colon= \min \{(\delta_0 /[\psi]_{C^{2,\alpha}(0)} )^{\frac{1}{1+\alpha}}, \delta_0/ \|D^2_{x'} \psi(0')\| \},$$ we get $\|D^2_{x'} \tilde{\psi}(0')\| = \eta \|D^2_{x'} \psi(0')\| \le \delta_0$ and
\\
\begin{eqnarray} 
[\psi_{\eta}]_{C^{2,\alpha}(0)} &=& \sup_{y'\in B'_1 \atop y'\not= 0} \frac{|D^2_{y'} \tilde{ \psi}_{\eta,\tau}(y')   - D^2_{y'} \tilde{\psi}_{\eta}(0)|}{|y'|^{\alpha}} \nonumber\\&\le&  \eta^{1+\alpha} \cdot [\psi]_{C^{2,\alpha}(0)} \le \delta_0.\label{Peq2}
\end{eqnarray}
Moreover
$$
\|D^2_{x'} \psi\|_{L^{\infty}(B'_1)} \le \sup_{x'\in B'_1} \|D^2_{x'} \psi(x') - D^2_{x'} \psi(0')\| + \|D^2_{x'} \psi(0)\| \le 2 \delta_0.
$$
\end{itemize}
 
For (ii) and (iii), we get,
\begin{eqnarray} 
[g_{\eta, \tau}]_{C^{1,\alpha}(0)} &=& \sup_{y \in B_1 \atop y \not= 0} \frac{|\nabla g_{\eta,\tau}(y) - \nabla g_{\eta,\tau}(0)|}{|y|^{\alpha}} 
= \frac{\eta^2}{\tau}  \sup_{y \in B_1 \atop y \not= 0} \frac{|\nabla g( \eta y) - \nabla g(0)|}{|y|^{\alpha}} \\ &\le& \frac{\eta^{2+\alpha}}{\tau}[g]_{C^{1,\alpha}(0)}. \label{Peq1}
\end{eqnarray}

Thus,  choosing
$$
\tau \colon= \|u\|_{L^{\infty}(B_1)} + \eta^{2+\alpha}\delta^{-1}_0 \cdot ([g]_{C^{1,\alpha}(0)} + \mathfrak{c}_{f^{-}} + \mathfrak{c}_{f^{+}}),
$$
we have $\|w\|_{L^{\infty}(B_1)} \le 1$.  Moreover, by \eqref{Peq0},
$$
	[g_{\eta,\tau}]_{C^{0,\alpha}(0)} + \mathfrak{c}_{f^{-}_{\eta,\tau}} + \mathfrak{c}_{f^{+}_{\eta,\tau}} \le \frac{\eta^{2+\alpha}}{\tau} \left([g]_{C^{0,\alpha}(0)} + \mathfrak{c}_{f^{-}} + \mathfrak{c}_{f^{+}} \right) \le \delta_0
$$
Thus, if we show that there exist quadratic polynomials  
$$\mathfrak{P}^{\pm}_{\eta,\tau}(y) =\frac{1}{2} y^t \cdot \mathcal{A}^{\pm}_{\eta,\tau} \cdot y + \mathcal{B}_{\eta,\tau} \cdot y + \mathfrak{b}_{\eta,\tau}$$ such that
$$
|w^{\pm}(y) - \mathfrak{P}^{\pm}_{\eta,\tau}(y)| \le \mathfrak{C}_{\eta,\tau} \cdot |y|^{2+\alpha}, \,\,\, \forall \, y \in \tilde{\Omega}_{1/2}
$$
and there exists $\mathfrak{C} >0$ depending only on $n,\lambda,\Lambda$, such that
$$
\|\mathcal{A}^{\pm}_{\eta,\tau}\|_{\textrm{Sym}(b)} + |\mathfrak{B}_{\eta,\tau}| + |\mathfrak{b}_{\eta,\tau}| \le \tilde{\mathfrak{C}} \|\psi_{\eta}\|_{C^{2,\alpha}(0)}\cdot (|g_{\eta,\tau}(0)| + [g_{\eta,\tau}]_{C^{0,\alpha}(0)} + \mathfrak{c}_{f^{-}_{\eta,\tau}} + \mathfrak{c}_{f^{+}_{\eta, \tau}}),
$$
then, rescaling back, we get that
$$
|u^{\pm}(x) - \mathfrak{P}^{\pm}(x)| \le \mathfrak{C}_0 |x|^{2+\alpha}, \,\,\, \forall \, x \in \Omega^{\pm}_{\frac{\eta}{2}}(y_0)= B_{\eta/2}(y_0) \cap \Omega^{\pm}
$$
with 
$$\mathfrak{P}^{\pm}(x) = \frac{1}{2} x^t \cdot \mathcal{A}^{\pm} \cdot x + \mathcal{B} \cdot x + \mathfrak{b},$$
where 
$$\mathcal{A}^{\pm} =\frac{\tau}{\eta^2} \mathcal{A}^{\pm}_{\eta,\tau}, \, \mathcal{B} = \frac{\tau}{\eta} \mathfrak{B}_{\eta, \tau}, \, \mathfrak{b} = \tau \mathfrak{b}_{\eta,\tau}  \quad \textrm{and} \quad \mathfrak{C}_0 = \frac{\tau}{\eta^{2+\alpha}} \mathfrak{C}_{\eta,\tau}$$ and
\begin{eqnarray*}
\frac{\eta^2}{\tau} \|\mathcal{A}^{\pm}\|_{\textrm{Sym}(n)} + \frac{\eta}{\tau} |\mathcal{B}| &+& \frac{1}{\tau} |\mathfrak{b}| \le \tilde{\mathfrak{C}} \|\psi_{\eta}\|_{C^{2,\alpha}(0)} \cdot \left( |g_{\eta,\tau}(0)| + [g_{\eta,\tau}]_{C^{0,\alpha}(0)} + \mathfrak{c}_{f^{-}_{\eta,\tau}} + \mathfrak{c}_{f^{+}_{\eta,\tau}}\right)\\
&\le& \tilde{\mathfrak{C}} \eta^{1+\alpha} \|\psi\|_{C^{2,\alpha}(0)} \cdot \left[ \frac{\eta}{\tau} |g(y_0)| + \frac{\eta^{2+\alpha}}{\tau} [g]_{C^{0,\alpha}(y_0)}+ \frac{\eta^{2+\alpha}}{\tau} (\mathfrak{c}_{f^{-}} + \mathfrak{c}_{f^{+}}) \right ].
\end{eqnarray*}
Multiplying by $\frac{\tau}{\eta^{2+\alpha}}$, and using the definition of $\eta$, we get
$$
|\mathcal{A}^{\pm}| + |\mathfrak{b}| + |\mathfrak{C}_0| \le \tilde{\mathfrak{C}} \delta^{-1}_0 [\psi]_{C^{2,\alpha}(0)} \cdot \left(|g(0)| + [g]_{C^{0,\alpha}(0)} + \mathfrak{c}_{f^+} + \mathfrak{c}_{f^{-}} \right).
$$

\subsection{Step 2: Iterative process}

In the sequel, we shall iterate Lemma \ref{LSS2} in appropriate dyadic balls as to obtain the precise sharp
oscillation decay of the difference between $u^{\pm}$ and quadratic polynomials functions. As a byproduct of a perturbation argument of \cite{MPA} we get the following result 

\begin{lemma}[Iterative process] \label{L02}
Suppose that the conditions of the previous lemma hold true. Given $k \ge 1$,  $F \in \mathcal{E}$$(\lambda,\Lambda)$ and  the Small Ellipticity Aperture condition $$\frac{\Lambda}{\lambda} -1 \le \delta,$$ it holds that there are quadratic polynomials 
$$
\mathfrak{P}_k(x) = \frac{1}{2} x^t \cdot \mathcal{A}^{\pm}_{k} \cdot x + \mathcal{B}^{\pm}_k +\mathfrak{c}_k
$$ 
satisfying $\textrm{Tr}(\mathcal{A}^{\pm}_k)=0$,
\begin{eqnarray}
	&&(\mathcal{B}^+_k)_i - (\mathcal{B}^{-}_k)_i = 0 \,\,\, \textrm{if} \,\,\, i < n \,\,\, \textrm{and} \,\,\, (\mathcal{B}^+_k)_n - (\mathcal{B}^{-}_k)_n = g(0) \label{C01}\\
&&(\mathcal{A}^{+}_k)_{ij} - (\mathcal{A}^{-}_k)_{ij} =0 \,\,\, \textrm{if} \,\,\, i,j < n, \,\,\, \textrm{and} \,\,\, (\mathcal{A}^{+}_k)_{nj}-(\mathcal{A}^{-}_k)_{nj} = g_{x_j}(0) \,\,\, \textrm{if} \,\,\, j < n \label{C02}
\end{eqnarray}
with
\begin{equation} \label{L12}
	\rho^{2(k-1)} \|\mathcal{A}^{\pm}_k - \mathcal{A}^{\pm}_{k-1}\|_{\textrm{Sym}(n)} + \rho^{k-1} |\mathcal{B}^{\pm}_k - \mathcal{B}^{\pm}_{k-1}| + |\mathfrak{c}_{k} - \mathfrak{c}_{k-1}| \le \mathfrak{C}_0 \cdot \rho^{(k-1)(2+\alpha)}
\end{equation}
with $\mathfrak{P}_0 \equiv 0$ and $\mathfrak{C}_0 = \mathfrak{C}_{0}(n,\lambda,\Lambda, \alpha)>0$, and such that
\begin{equation} \label{L13}
\|u^{\pm} - \mathfrak{P}^{\pm}_{k}\|_{L^{\infty}(\Omega^{\pm}_{\rho^k})} \le \rho^{k(2+\alpha)}
\end{equation}
\end{lemma}
\begin{proof} Since the proof depends on  Lemma \ref{LSS2}, we will omit some steps. The complete proof can be found in \cite{MPA}.  We will only outline the case $D^2_{x'} \psi (0')=0$ because the case $g(0)=0$ is analogous.   The proof proceeds by induction. For $k=1$, by the normalization, we are under the assumptions of Lemma  is precisely the statement of Lemma \ref{LSS2}.

Suppose we have verified \eqref{C01}, \eqref{C02}, \eqref{L12} and \eqref{L13} for $k=1,2,\ldots, \ell$.  Consider the rescaled function $v : \overline{B_1} \rightarrow \mathbb{R}$
$$
	v(x) = \frac{u(\rho^{\ell} x) - \mathfrak{P}_{\ell}(\rho^{\ell} x) }{\rho^{\ell (2+\alpha)}},
$$
where $\mathfrak{P}(x) \colon= \mathfrak{P}^{+}_{\ell} \chi_{\tilde{\Omega}^{+}_{\ell}} + \mathfrak{P}^{-}_{\ell} \chi_{\tilde{\Omega}^{-}_{\ell}}$,
$
\tilde{\Omega}^{\pm}_{\ell} = \{x \in B_1 : \rho^{\ell} x \in \Omega^{\pm}\}$ and  $ \tilde{\Gamma}_{\ell} = \{x \in B_1 : x_n = \psi_{\ell}(x')\},
$
for $\psi_{\ell}(x') = \rho^{-\ell} \psi(\rho^{\ell} x')$.  By the induction hypothesis, $\|v\|_{L^{\infty}(B_1)} \le 1$, and $v$ satisfies
$$
 \left\{
\begin{array}{rclcl}
 F_{\ell}(D^2 v, x) &=& f^{\pm}_{\ell}& \mbox{in} &   \tilde{\Omega}^{\pm}_{\ell} \\
 v^+_{\nu_{\ell}} - v^{-}_{\nu_{\ell}}  &=& g_{\ell} &\mbox{on}& \tilde{\Gamma}_{\ell}
\end{array}
\right.
$$
in the viscosity sense, with
\begin{eqnarray*}
F_{\ell}(M,x) &=& \rho^{-\ell \alpha} F_{\tau}(\rho^{\ell \alpha} M + \mathcal{A}^{\pm}_{\ell}, \rho^{\ell} x) \,\,\,\, \textrm{for} \,\,\, M \in \textrm{Sym}(n), \, x \in B_1.\\
f^{\pm}_{\ell}(x) &=& \rho^{- \ell \alpha} f^{\pm}(\rho^{\ell} x) \,\,\, \textrm{for any} \,\,\, x \in \tilde{\Omega}^{\pm}_{\ell}\\
g_{\ell}(x) &=& \rho^{-\ell (1+\alpha)} \left[ g(\rho^{\ell} x) - (\mathcal{A}^{+}_{\ell} - \mathcal{A}^{-}_{\ell}) (\rho^{\ell} x) \cdot \nu_{\ell}(x)-(\mathcal{B}^{+}_{\ell} -\mathcal{B}^{-}_{\ell}) \cdot \nu_{\ell}(x) \right] \,\,\, \textrm{for} \,\,\, x \in \tilde{\Gamma}_{\ell}
\end{eqnarray*}
It's possible to find $\delta>0$ such that,  $v$ satisfies the assumptions of Lemma \ref{LSS2}.  In fact,  the proof follows exactly as the one of proof of Theorem 1.3 in \cite{MPA}.  Then, there are  quadratic polynomials
$$
\mathfrak{P}^{\pm}(x) = \frac{1}{2}x^t \cdot \mathcal{A}^{\pm} \cdot x + \mathcal{B} \cdot x + \mathfrak{c},
$$
where $\mathcal{A}^{\pm} \in \textrm{Sym}(n)$, $\mathcal{B}^{\pm} \in \mathbb{R}^n$, $\mathfrak{c} \in \mathbb{R}$, and $
\|\mathcal{A}^{\pm}\|_{\textrm{Sym}(n)} + |\mathcal{B}^{\pm}| + |\mathfrak{c}| \le \mathfrak{C}_0, 
$
satisfying
\begin{eqnarray}
	\mathcal{B}^+_i - \mathcal{B}^{-}_i = 0 \,\,\, \textrm{if} \,\,\, i < n \quad &\textrm{and}& \quad \mathcal{B}^+_n - \mathcal{B}^{-}_n = g(0) \label{S03}\\
\mathcal{A}^{+}_{ij} - \mathcal{A}^{-}_{ij} =0 \,\,\, \textrm{if} \,\,\, i,j < n \,\,\, &\textrm{and}& \,\,\, \mathcal{A}^{+}_{nj}-\mathcal{A}^{-}_{nj} = g_{x_j}(0) \,\,\, \textrm{if} \,\,\, j < n \label{S04}
\end{eqnarray}
such that 
\begin{equation} \label{C05}
\|v^{\pm} - \mathfrak{P}^{\pm}\|_{L^{\infty}(\Omega^{\pm}_{\rho})} \le \rho^{2+\alpha}.
\end{equation}
Rewriting \eqref{C05} back to the original domain yields
$$
\left | u^{\pm}(z) - \mathfrak{P}^{\pm}_{\ell}(z) - \rho^{\ell(2+\alpha)} \mathfrak{P}^{\pm}_{\ell}(\rho^{-\ell} z) \right | \le \rho^{(\ell +1)(2+\alpha)}
$$
Finally, by defining $\mathfrak{P}^{\pm}_{\ell +1}(x) \colon= \mathfrak{P}^{\pm}_{\ell}(x) + \rho^{\ell (2+\alpha)} \mathfrak{P}^{\pm}(\rho^{-\ell} x)$ we check the $(\ell+1)^{\underline{th}}$  step of induction. Note that
$$
\mathcal{A}^{\pm}_{\ell +1} = \mathcal{A}^{\pm}_{\ell} + \rho^{\ell \alpha} \mathcal{A}^{\pm}, \,\,\,\, \mathcal{B}^{\pm}_{\ell +1} \colon= \mathcal{B}^{\pm}_{\ell} + \rho^{\ell (1+\alpha)}, \, \textrm{and} \,\,\, \mathfrak{c}_{\ell +1} \colon= \mathfrak{c}_{\ell} + \rho^{\ell (2+\alpha)} \mathfrak{c}. 
$$
Thus, using that $\|\mathcal{A}^{\pm}\|_{\textrm{Sym}(n)} + |\mathcal{B}^{\pm}| + |\mathfrak{c}| \le \mathfrak{C}_0$ we get that
$$
\rho^{2 \ell} \|\mathcal{A}^{\pm}_{\ell +1} - \mathcal{A}^{\pm}_{\ell}\|_{\textrm{Sym}(n)} + \rho^{\ell}|\mathcal{B}^{\pm}_{\ell +1} - \mathcal{B}^{\pm}_{\ell}| + |\mathfrak{c}_{\ell+1} - \mathfrak{c}_{\ell}| \le \mathfrak{C}_0 \rho^{\ell (2+\alpha)}.
$$
 Moreover, by \eqref{S03} and \eqref{S04} we have 
\begin{eqnarray*}
(\mathcal{B}^+_{\ell+1})_i - (\mathcal{B}^{-}_{\ell+1})_i &=& (\mathcal{B}^+_{\ell})_i - (\mathcal{B}^{-}_{\ell})_i  + \rho^{\ell(2+\alpha)} (\mathcal{B}^+_{i}-\mathcal{B}^{-}_i)=0, \,\, \textrm{if} \,\, i < n;\\
\mathcal{B}^+_{\ell+1})_n - (\mathcal{B}^{-}_{\ell+1})_n &=& (\mathcal{B}^+_{\ell})_n - (\mathcal{B}^{-}_{\ell})_n  + \rho^{\ell(2+\alpha)} (\mathcal{B}^+_{n}-\mathcal{B}^{-}_n)=g(0);\\
(\mathcal{A}^{+}_{\ell+1})_{ij} - (\mathcal{A}^{-}_{\ell+1})_{ij} &=& (\mathcal{A}^{+}_{\ell})_{ij} - (\mathcal{A}^{-}_{\ell})_{ij} + \rho^{\ell \alpha}(\mathcal{A}^{+}_{ij} - \mathcal{A}^{-}_{ij})=0 \,\, \textrm{if} \,\, i,j < n;\\
(\mathcal{A}^{+}_{\ell+1})_{nj} - (\mathcal{A}^{-}_{\ell+1})_{nj} &=& (\mathcal{A}^{+}_{\ell})_{nj} - (\mathcal{A}^{-}_{\ell})_{nj} + \rho^{\ell \alpha}(\mathcal{A}^{+}_{nj} - \mathcal{A}^{-}_{nj})=g_{x_j}(0) \,\, \textrm{if} \,\, j < n;\\
\textrm{Tr}(\mathcal{A}^{\pm}_{\ell+1}) &=& \textrm{Tr}(\rho^{\ell \alpha} \mathcal{A}^{\pm} + \mathcal{A}^{\pm}_{\ell}) = \rho^{\ell \alpha}\textrm{Tr}(\mathcal{A}^{\pm})=0.
\end{eqnarray*}
This concludes the proof of the Theorem.

\end{proof}

With Lemma \ref{L02} in hand, we now prove our main result Theorem \ref{BMO}. 
\begin{proof}[Proof of Theorem \ref{BMO}]
It is enough to prove the estimate at $x=0$.  Now, we notice that if follows  from \eqref{L12}, namely $\mathcal{A}^{\pm}_{k}$, $\mathcal{B}^{\pm}_{k}$ and $\mathfrak{c}_{k}$, are Cauchy sequences in $\textrm{Sym}(n)$, $\mathbb{R}^n$ and in $\mathbb{R}$, respectively. Let $\mathcal{A}^{\pm}_{\infty}$, $\mathcal{B}^{\pm}_{\infty}$ and $\mathfrak{c}_{\infty}$ be the limiting coefficients, i.e.
$$
\mathcal{A}^{\pm}_{\infty} = \lim_{k \to +\infty} \mathcal{A}^{\pm}_k, \, \mathcal{B}^{\pm}_{\infty} = \lim_{k \to +\infty} \mathcal{A}^{\pm}_{k} \quad \textrm{and} \quad \mathfrak{c}_{\infty} = \lim_{k \to +\infty} \mathfrak{c}_k.
$$
Ij the sequel, in view \eqref{L12}, 
$$
\| \mathfrak{P}^{\pm}_k - \mathfrak{P}^{\pm}_{\infty}\|_{L^{\infty}(\Omega^{\pm}_{\rho})} \le C_{\star} \rho^{2+\alpha}
$$
for some universal constant $C_{\star}>0$.  Now fix $0 < r < \rho$, we choose $k \in \mathbb{N}$ such that
$$
\rho^{k+1} < r \le \rho^k.
$$ 
We estimate
\begin{eqnarray*}
\|u^{\pm} - \mathfrak{P}^{\pm}_{\infty}\|_{L^{\infty}(\Omega^{\pm}_{r})} \le  \mathfrak{C}_0 \cdot r^{2+\alpha}
\end{eqnarray*}
\end{proof}

\section{Schauder estimates of quasiconcave equations} \label{Quase}

 In this section  we will study Schauder estimates for viscosity solution of fully nonlinear elliptic transmission problems when the operator fails to be concave or convex in the space of symmetric matrices. In particular, it is assumed that either the level sets are convex or the operator is concave, convex or close to a linear function near infinity.

 \subsection{Regularity estimates up to the flat interface}
 
 In this part, we will work initially with the regularity estimates up to the flat interface with constant coefficients, that is, with the following problem
  \begin{equation} \label{up}
	\left\{
	\begin{array}{rclcl}
		F(D^2 u) &=& f^{\pm} & \mbox{in} & B^{\pm}_1,\\
		u^{+}_{x_n} - u^{-}_{x_n}&=& g & \mbox{on} & T
	\end{array}
	\right.
	\end{equation}
 where $F$ is a quasiconvex operator.

 \begin{theorem} \label{flat}
 Let $F$ be a a quasiconcave operator and fix $\alpha \in (0,\alpha_0)$. Assume that $g \in C^{1,\alpha}(T)$ and $f^{\pm} \in C^{0,\alpha}(\overline{B^{\pm}_{1}})$ with $f^+ = f^{-}$ on $T$. Then any bounded viscosity solution $u$ of the problem
 $$
	\left\{
	\begin{array}{rclcl}
		F(D^2 u) &=& f^{\pm} & \mbox{in} & B^{\pm}_1,\\
		u^{+}_{x_n} - u^{-}_{x_n}&=& g & \mbox{on} & T,
	\end{array}
	\right.
	$$
satisfies $u \in C^{2,\alpha}(\overline{B^{\pm}_{1/2}})$ and
	$$
	\|u\|_{C^{2,\alpha}(\overline{B^{\pm}_{1/2}})} \le C \cdot (\|u\|_{L^{\infty}(B_1)} + \|g\|_{C^{1,\alpha}(T)} + \|f^{-}\|_{C^{0,\alpha}(\overline{B^{-}_1})} + \|f^{+}\|_{C^{0,\alpha}(\overline{B^{+}_1})}),
	$$
where $C>0$ is a constant depending only on $n,\lambda$ and $\Lambda.$
 \end{theorem}

 To prove Theorem \ref{flat} we will need the following result.
 
 \begin{proposition} \label{prop}
  Let $F$ be a quasiconcave operator, and  $0 < r \le 1$. Suppose that $\mathfrak{h}$ is a bounded viscosity solution of the problem
   \begin{equation} \label{flat1}
	\left\{
	\begin{array}{rclcl}
		F(D^2 \mathfrak{h}) &=& 0 & \mbox{in} & B^{\pm}_1,\\
		\mathfrak{h}^{+}_{x_n} - \mathfrak{h}^{-}_{x_n}&=& 0 & \mbox{on} & T.
	\end{array}
	\right.
	\end{equation}
	Then, for any radii $0 < \rho \le \frac{r}{2}$, it holds that  $\mathfrak{h} \in C^{2,\alpha_0}(\overline{B^{\pm}_{\rho}})$ with the following estimates holding true
	\begin{eqnarray*}
	\displaystyle\osc_{B^{\pm}_{\rho}} \left( \mathfrak{h}^{\pm} - \frac{1}{2} x^t \cdot D^2 \mathfrak{h}^{\pm}(0) \cdot x - \nabla \mathfrak{h}(0) \cdot x\right)& \le& C \left( \frac{\rho}{r}\right)^{2+\alpha_0} \cdot \osc_{\overline{B_r}} \, \mathfrak{h}\\
	r^2 \|D^2 \mathfrak{h}^{\pm}(0)\|_{\textrm{Sym}(n)} + r |\nabla \mathfrak{h}(0) |&\le& C \cdot \osc_{\overline{B_r}} \, \mathfrak{h},
	\end{eqnarray*}
	where $C$ is a constants that depends only on $ n,\lambda$, and $ \Lambda$.
 \end{proposition}
 \begin{proof}

 Fix $0 < \rho \le \frac{r}{2}$. Let $w$ be the unique viscosity solution to 
 $$
  \left\{
	\begin{array}{rclcl}
		F(D^2 w) &=& 0 & \mbox{in} & B_1,\\
		w &=& \mathfrak{h} & \mbox{on} & \partial B_1.
	\end{array}
	\right.
	$$
 Since $F$ is a quasiconcave, by Theorem \ref{usar}, $w \in C^{2,\alpha_0}_{loc}(B_1)$ and
 $$
 \|w\|_{C^{2,\alpha_0}(\overline{B_{1/2}})} \le C \|w\|_{L^{\infty}(B_1)}.
 $$
 Moreover, by uniqueness of viscosity solutions to flat interface problems, it follows that $w = \mathfrak{h}$ in $B_1$. Now, applying the above result for $\mathfrak{h}(rx)$, for $x \in B_1$, we get $\mathfrak{h} \in C^{2,\alpha_0}(\overline{B_{1/2}})$.  Let $v(x)= \mathfrak{h}(rx)-\mathfrak{h}(0)$, for $x \in B_1$. Then, by the mean value theorem,
	\begin{eqnarray*}
	\osc_{B^{\pm}_{\rho}} \left( \mathfrak{h}^{\pm} - \frac{1}{2} x^t \cdot D^2 \mathfrak{h}^{\pm}(0) \cdot x - \nabla \mathfrak{h}(0) \cdot x\right) &\le& C \rho^{2+\alpha_0}  \|D^2 \mathfrak{h}^{\pm}\|_{C^{0,\alpha_0}(\overline{B^{\pm}_{\rho}})}\\& =& C \left( \frac{\rho}{r}\right)^{2+\alpha_0} [D^2 v^{\pm}]_{C^{0,\alpha_0}(\overline{B^{\pm}_{\rho/r}})} \\
	&\le& C \left(\frac{\rho}{r} \right)^{2+\alpha_0} \|v\|_{L^{\infty}(B_1)} \\
	&\le& C  \left(\frac{\rho}{r} \right)^{2+\alpha_0} \osc_{\overline{B_r}} \mathfrak{h}.
	\end{eqnarray*}
	Moreover,
	$$
	r^2 \|D^2 \mathfrak{h}^{\pm}(0)\|_{\textrm{Sym}(n)} + r |\nabla \mathfrak{h}(0)| \le C \|v\|_{L^{\infty}(B_1)} \le  C \osc_{\overline{B_r}} \mathfrak{h}.
	$$
  This concludes the proof of the Proposition.
 \end{proof}
 
 \begin{proof}[Proof of Theorem \ref{flat}]
 It is enough to prove the estimate at $x=0$. Without loss of generality, assume that $v(0)=0$ and $g(0)=0$. Suppose further that $f^{+}(0)=f^{-}(0)=0$. Let
 $$
 M= \|u\|_{L^{\infty}(B_1)} + \|g\|_{C^{1,\alpha}(T)} + \|f^{-}\|_{C^{1,\alpha}(\overline{B^{-}_1})} + \|f^{+}\|_{C^{1,\alpha}(\overline{B^{+}_1})}.
 $$

 {\bf Claim:} \, For all $k \ge 0$, there exist  $0 < \gamma < 1$ and $\mathfrak{C}_0, \mathfrak{C}_1 >0$, depending only on $n,\lambda, \Lambda$ and $\alpha$, and sequences of quadratic polynomials $\mathfrak{P}^{\pm}_k(x) = \frac{1}{2} x^t \cdot \mathcal{A}^{\pm}_k \cdot x + \mathcal{B}_k \cdot x, \,\, k \ge 0 $, such that
 \begin{eqnarray}
 \osc_{B_{\rho^k}} \left(u^{\pm} - \mathfrak{P}^{\pm} \right) &\le& \mathfrak{C}_0 M \gamma^{k(2+\alpha)} \label{flat3}\\
 \gamma^{k-1} \|\mathcal{A}^{\pm}_k - \mathcal{A}^{\pm}_{k-1}\|_{\textrm{Sym}(n)} + |\mathcal{B}_{k}-\mathcal{B}_{k-1}| &\le& \mathfrak{C}_1 M \gamma^{(k-1)(2+\alpha)}, \label{flat4}
 \end{eqnarray}
 for any $k \ge 0$, where $\mathcal{A}^{\pm}_{-1} =0$ and $\mathcal{B}_{-1}=0$. Furthermore, $F(\mathcal{A}^{\pm},0)=0$ and
 \begin{equation} \label{flat5}
 (\mathcal{A}^{+}_{k})_{ij} - (\mathcal{A}^{-}_{k})_{ij}=0 \,\,\, \textrm{if} \,\, i,j \not= n \quad \textrm{and} \quad (\mathcal{A}^{+}_{k})_{jn} - (\mathcal{A}^{-}_{k})_{jn}- g_{x_j}(0) \,\,\, \textrm{if} \,\, j \not= n,
 \end{equation}
 where $\mathcal{A}_{ij}$ denotes the element in the $(i,j)$-entry of the matrix $\mathcal{A}$.
 
 We argue by finite induction.  In fact, for $k=0$, choose $\mathcal{B}_0=0$ and $\mathcal{A}^{\pm}_0$ symmetric sucht that \eqref{flat5} holds and $(\mathcal{A}^{\pm}_0)_{nn}$ is given by $F(\mathcal{A}^{\pm}_0)=0$. Then $\|\mathcal{A}^{\pm}_0\| \le \mathfrak{C}_1 |\nabla' g(0)| \le \mathfrak{C}_1 M $, for some $\mathfrak{C}_1 >0$. Moreover, choose $\mathfrak{C}_0 >1$  large so that
 $$
 	\osc_{B_1} \left( u^{\pm}- \mathfrak{P}^{\pm}_0 \right) \le 2 \left( \|u\|_{L^{\infty}(B_1)} + \|\mathcal{A}^{\pm}_0\|_{\textrm{Sym}(n)}\right) \le 2(1+\mathfrak{C}_1) M \le \mathfrak{C}_0 M.
 $$
 Assume that \eqref{flat3}-\eqref{flat5} hold for some $k \ge 0$. Let $r=\rho^k_0$ and $\mathfrak{P}_k = \mathfrak{P}^+_k \chi_{\overline{B^+_r}} + \mathfrak{P}^{-}_k \chi_{\overline{B^{-}_r}}$.  Note that, by \eqref{flat5}, we have $\mathfrak{P}^{+}_k = \mathfrak{P}^{-}_k$ on $T \cap B_r$, and thus, $\mathfrak{P}_k \in C(\overline{B_r})$. Let $v \in C(\overline{B_r})$ be the viscosity solution to
 $$
  \left\{
	\begin{array}{rclcl}
		F(D^2 v + \mathcal{A}^{\pm}_k) &=& 0 & \mbox{in} & B^{\pm}_r,\\
		v^+_{x_n} - v^{-}_{x_n} &=& 0 &\mbox{on}& T \cap B_r\\
		v &=& u - \mathfrak{P}_k & \mbox{on} & \partial B_r
	\end{array}
	\right.
 $$
For the proof of existence, see \cite[Theorem 4.11]{MPA}. From ABP estimate (cf. \cite[Theorem 2.1]{MPA}) and the fact that $F(\mathcal{A}^{\pm}_k,0)=0$,
 \begin{equation} \label{flat6}
 \osc_{\overline{B_r}} \, v \le \osc_{\overline{B_r}} \, (u-\mathfrak{P}_k).
 \end{equation}
 Fix $\rho \le \frac{r}{2}$ to be determined later. By Proposition \ref{prop}, we know that $v^{\pm} \in C^{2,\alpha_0}(\overline{B^{\pm}_{\rho}})$, with
 \begin{eqnarray}
 \osc_{B^{\pm}_{\rho}} \left( v^{\pm} - \mathfrak{P}^{\pm}\right) &\le& C \left( \frac{\rho}{r}\right)^{2+\alpha_0} \osc_{\overline{B_r}} v\label{flat7}\\
 r^2 \|\mathcal{Q}^{\pm}\|_{\textrm{Sym}(n)} + r |\mathcal{R}| &\le& C \osc_{\overline{B_r}} v, \label{flat8}
 \end{eqnarray}
 where $\mathfrak{P}^{\pm}(x) = \frac{1}{2} x^t \cdot \mathcal{Q}^{\pm} \cdot x + \mathcal{R} \cdot x $, $\mathcal{Q}^{\pm} = D^2 v^{\pm}(0)$, and $\mathcal{R} = \nabla v(0)$.  Moreover, as $F(D^2 v + \mathcal{A}^{\pm}_k)=0$, it follows that 
 \begin{equation} \label{flat9}
 F(\mathcal{Q} + \mathcal{A}^{\pm}_k)=0.
 \end{equation}
 Let $\rho = \gamma r$ and $\epsilon = \alpha_0-\alpha>0$. Choose $\gamma \le 1/2$ small enough so that $C \gamma^{\epsilon} \le 1/2$. Combining \eqref{flat6}, \eqref{flat7}, and the induction hypothesis, we see that
 \begin{equation} \label{flat10}
 \osc_{B^{\pm}_{\rho}} \left(v^{\pm}-\mathfrak{P}^{\pm} \right) \le C \left( \frac{\rho}{r}\right)^{2+\alpha_0} \osc_{\overline{B_r}} \left( u - \mathfrak{P}_k\right) \le \frac{1}{2} \mathfrak{C}_0 M \gamma^{(k+1)(2+\alpha)}.
 \end{equation}
 Let $w=u-\mathfrak{P}_k-v$. So, taking into account \eqref{flat5}, it follows that
 $$
   \left\{
	\begin{array}{rclcl}
		w \in \mathcal{S}(\lambda/n, \Lambda, f^{\pm}) &&  & \mbox{in} & B^{\pm}_r,\\
		w^+_{x_n} - w^{-}_{x_n} &=& g - \nabla'g(0) \cdot x'  &\mbox{on}& T \cap B_r\\
		w &=& 0 & \mbox{on} & \partial B_r
	\end{array}
	\right.
 $$Now, using the rescaled ABP estimate (see \cite[Theroem 2.1]{MPA}), and the assumptions on $g$ and $f^{\pm}$, we get
 \begin{eqnarray*}
 	\|w\|_{L^{\infty}(B_{\rho})} &\le& C \rho \left( \|g-g(0)-\nabla' g(0) \cdot x'\|_{L^{\infty}(T \cap B_{\rho})} + \|f^{\pm}-f^{\pm}(0)\|_{L^n(B^{\pm}_{\rho})} \right)\\
	&\le& C \rho^{2+\alpha} \left( [g]_{C^{1,\alpha}(0)} + [f^{\pm}]_{C^{0,\alpha}(0)} \right) \le C M \rho^{2+\alpha}.
 \end{eqnarray*}
 Choose $\mathfrak{C}_0 \ge 4C$. In view of \eqref{flat10} and the previous estimate, we have
 $$
 \osc_{B^{\pm}_{\gamma^{k+1}}} \left( u^{\pm} - \mathfrak{P}^{\pm}_k - \mathfrak{P}^{\pm} \right) \le \osc_{B_{\rho}} \, w + \osc_{B^{\pm}_{\rho}} \, (v^{\pm} - \mathfrak{P}^{\pm}) \le \mathfrak{C}_0 M \gamma^{(k+1)(2+\alpha)}.
 $$
 Hence, the estimate in \eqref{flat3} holds for $k+1$ with $\mathfrak{P}_{k+1} = \mathfrak{P}_k + \mathfrak{P}$. To prove \eqref{flat4}, we use \eqref{flat8}, \eqref{flat6}, and the induction hypothesis to get
 $$
 \gamma^k \|\mathcal{A}^{\pm}_{k+1}\mathcal{A}^{\pm}_k\|_{\textrm{Sym}(n)} + | \mathcal{B}_{k+1} - \mathcal{B}_k| \le \mathfrak{C}_1 M \gamma^{k(1+\alpha)}, 
 $$
 where $\mathfrak{C}_1 = C \mathfrak{C}_0$. Since, $\mathfrak{A}^{\pm}_{k+1} = \mathcal{A}^{\pm}_k + \mathcal{Q}$, then $F(\mathcal{A}^{\pm}_{k+1},0)=0$ and hence \eqref{flat5} follows from the induction hypothesis.  
 \end{proof}

 \subsection{Regularity estimate for non-flat interface problems}
 
  The key ingredients to prove Theorem \ref{BMOQ} are the following lemmas, whose proofs will be omitted because they are similar to those of \cite[Lemma 5.7 and 5.8]{MPA}.

 \begin{lemma}\label{L1}
Assume that $\Gamma= B_1 \cap \{x_n = \psi(x')\}$ for some $\psi \in C^2(B'_1)$. Fix a vector $D' \in \mathbb{R}^{n-1}$ and $b \in \mathbb{R}$.  Then, given any $\epsilon>0$ there is $\delta >0$ such that if $u \in C(B_1)$ is a viscosity solution of the problem
$$
	\left\{
	\begin{array}{rclcl}
		F(D^2 u) &=& f^{\pm} & \mbox{in} & \Omega^{\pm},\\
		u^{+}_{\nu} - u^{-}_{\nu}&=& g & \mbox{on} & \Gamma,
	\end{array}
	\right.
	$$
with $\|u\|_{L^{\infty}(B_1)} \le 1$ and 
\begin{equation} \label{Hip}
	\max \{\|\psi\|_{C^2(B'_1)}, \|g- D' \cdot x -b\|_{L^{\infty}(\Gamma)} , \|f^{\pm}\|_{L^{\infty}(\Omega^{\pm})} \} \le \delta,
\end{equation}
 then there exist a bounded viscosity solution $v \in C(B_{3/4})$ to 
 \begin{equation}\label{Laplace}
	\left\{
	\begin{array}{rclcl}
		F(D^2 v) &=& 0 & \mbox{in} & B^{\pm}_{3/4},\\
		v^{+}_{x_n} - v^{-}_{x_n}&=& D'\cdot x'+b & \mbox{on} & B_{3/4} \cap \{x_n=0\},
	\end{array}
	\right.
	\end{equation}
	satisfying
$$
\|u - v\|_{L^{\infty}(B_{3/4})} \le \epsilon.
$$
\end{lemma}

 \begin{lemma}\label{L_1}
Consider $\Gamma= B_1 \cap \{x_n = \psi(x')\}$, for some  function$\psi \in C^2(B'_1)$, and that the term $f^{\pm}$ satisfy the inequality 
$$
\left( \intav{B_r(x_0) \cap \Omega^{\pm}} |f^{\pm}|^n dx \right)^{1/n} \le C_{f^{\pm}} \cdot r^{\alpha-1},
$$
for all small radii $r>0$, $x_0 \in \Omega^{\pm} \cup \Gamma$, and some $\alpha \in (0,1)$. Fix $D' \in \mathbb{R}^{n-1}, b \in \mathbb{R}$ and consider $u\in C(B_1 \setminus \Gamma) \cap L^{\infty}(B_1)$  a viscosity solution to
$$
	\left\{
	\begin{array}{rclcl}
		F(D^2 u) &=& f^{\pm} & \mbox{in} & \Omega^{\pm},\\
		u^{+}_{\nu} - u^{-}_{\nu}&=& g & \mbox{on} & \Gamma
	\end{array}
	\right.
	$$
with $\|u\|_{L^{\infty}(B_1)} \le 1$, $u^+ -u^{-} \equiv h \in C^2(\Gamma)$.  Then, it holds that given an arbitrary  $\epsilon>0$, there exists $\tilde{\delta}>0$ such that if the inequality
\begin{equation} \label{Estima_1}
	\max \{\|\psi\|_{C^2(B'_1)}, \|g- D' \cdot x -b\|_{L^{\infty}(\Gamma)} , \|f^{\pm}\|_{L^{\infty}(\Omega^{\pm})} \} \le \tilde{\delta},
\end{equation}
holds, then there exists a bounded viscosity solution $v \in C(B_{1/2})$ to the problem
 \begin{equation}\label{Laplace_1}
	\left\{
	\begin{array}{rclcl}
		F(D^2 v) &=& 0 & \mbox{in} & B^{\pm}_{1/2},\\
		v^{+}_{x_n} - v^{-}_{x_n}&=& D'\cdot x'+b & \mbox{on} & B_{3/4} \cap \{x_n=0\}
	\end{array}
	\right.
	\end{equation}
satisfying
$$
\|u-v\|_{L^{\infty}(B_{1/2})} \le \epsilon.
$$

\end{lemma}

Theorem \ref{BMOQ} will be a consequence of iterating the next lemma.

\begin{lemma} \label{L_2}
Let $F$ be a quasiconcave and uniformly elliptic operator. Given $\alpha \in (0,\alpha_0)$, there are constants $\mathfrak{C}_0>0$, and $0 < \epsilon , \rho < 1/2$, depending only on the quantities $n,\lambda,\Lambda$, and $\alpha$ such that, if $u$ satisfies the assumptions of Lemma \ref{L_1} with $D' = \nabla' g(0)$ and $b=g(0)$, and $\|g\|_{C^1(\Gamma)} \le 1$, then there are  polynomials with quadratic growth
$$
\mathfrak{P}^{\pm}(x) = \frac{1}{2}x^t \cdot \mathcal{A}^{\pm} \cdot x + \mathfrak{B}^{\pm} \cdot x + \mathfrak{c},
$$
where $\mathcal{A}^{\pm} \in \textrm{Sym}(n)$, $\mathcal{B}^{\pm} \in \mathbb{R}^n$, $\mathfrak{c} \in \mathbb{R}$, and $\|\mathcal{A}^{\pm}\|_{\textrm{Sym}(n)} + |\mathcal{B}^{\pm}| + |\mathfrak{c}| \le \mathfrak{C}_0$ such that the following estimate holds true
$$
\|u^{\pm} - \mathfrak{P}\|_{L^{\infty}(\Omega^{\pm}_{\rho})} \le \rho^{2+\alpha}.
$$
Moreover, it holds that $F(\mathcal{A}^{\pm},0)=0$ and
\begin{eqnarray}
	\mathcal{B}^+_i - \mathcal{B}^{-}_i = 0 \,\,\, \textrm{if} \,\,\, i < n \quad &\textrm{and}& \quad \mathcal{B}^+_n - \mathcal{B}^{-}_n = g(0), \label{S_3}\\
\mathcal{A}^{+}_{ij} - \mathcal{A}^{-}_{ij} =0 \,\,\, \textrm{if} \,\,\, i,j < n \,\,\, &\textrm{and}& \,\,\, \mathcal{A}^{+}_{nj}-\mathcal{A}^{-}_{nj} = g_{x_j}(0) \,\,\, \textrm{if} \,\,\, j < n \label{S_4}
\end{eqnarray}

\end{lemma}

 \begin{proof}
   The proof includes the same lines as \cite[Lemma 7.1]{MPA} along with minor changes. The main difference is to replace Theorem 4.16 in \cite{MPA}  with Theorem \ref{flat}.
      \end{proof}

\begin{proof}[Proof of Theorem \ref{BMOQ}] The proof is omitted because it follows similar lines as the proof of \cite[Theorem 1.3]{MPA}. Indeed, the key ingredients are the Theorem \ref{flat} and Lemma \ref{L_2}. Details are left to the interested reader. 
\end{proof}

\section{Applications and comments} \label{SecCasoGeral}

\subsection{Global estimates}

As a product of a translation argument, we obtain the following result: a global $C^{2,\alpha}$ estimate for Theorems \ref{BMO} and \ref{BMOQ}.
\begin{corollary}\label{Obs1}
\end{corollary}
\begin{enumerate}
\item \textbf{(First Case)} Fix any $\alpha \in (0,1)$. Assume that $g \in C^{1,\alpha}(\Gamma)$ and $f^{\pm} \in C^{0,\alpha}(\overline{\Omega}^{\pm})$ with $f^+ = f^{-}$ on $\Gamma$. 
 Then there exists $\delta>0$ such that if \eqref{Cor} is satisfied then any bounded viscosity solution $u$ to the problem \eqref{E1} satisfies $u \in C^{2,\alpha}(\overline{\Omega^{\pm}_{1/2}})$ with estimate
$$
\|u^{\pm}\|_{C^{2,\alpha}(\overline{\Omega^{\pm}_{1/2}})} \le C_0  \|\psi\|_{C^{2,\alpha}(\overline{B'_1})}\left( \|u\|_{L^{\infty}(B_1)} +  \|g\|_{C^{1,\alpha}(\Gamma)} + \|f^{\pm}\|_{C^{0,\alpha}(\overline{\Omega^{\pm}})}  \right).
$$
where $C_0>0$ depends only on $n,\Lambda, \lambda$ and $\alpha$.
\item \textbf{(Second Case)} Let $F \in \mathcal{E}(\lambda,\Lambda)$ be a quasiconvex or quasiconcave operator. Fix any $\alpha \in (0,\alpha_0)\footnote{$\alpha_0 \in (0,1)$ as in \eqref{Hom1}.}$. Assume that $g \in C^{1,\alpha}(\Gamma)$ and $f^{\pm} \in C^{0,\alpha}(\overline{\Omega}^{\pm})$ with $f^+ = f^{-}$ on $\Gamma$ . Then any bounded viscosity solution $u$ to the problem \eqref{E1} satisfies $u \in C^{2,\alpha}(\overline{\Omega^{\pm}_{1/2}})$ with the same estimate as in item $(1)$.
\end{enumerate}
\begin{proof}
   In fact, for each $x_0 \in \overline{\Omega^{\pm}_{1/2}}$,  define $w : B_1 \rightarrow \mathbb{R}$ by  $w(x) = 4 \cdot u\left( x_0 + \frac{1}{2}x\right)$ and $w^{\pm} \colon= w \Big |_{\Omega^{\pm}}$.  Then, $w \in C(B_1)$ is a viscosity solution to
   $$
   \left\{
	\begin{array}{rclcl}
		F(D^2 w) &=& \tilde{f}^{\pm} & \mbox{in} & \Omega^{\pm},\\
		w^{+}_{\nu} - w^{-}_{\nu}&=& \tilde{g} & \mbox{on} & \Gamma
	\end{array}
	\right.
   $$ 
   where $\tilde{f}^{\pm}(x) \colon= f \left( x_0 + \frac{1}{2} x\right)$ and $\tilde{g}(x) \colon= 2 g \left(x_0 + \frac{1}{2}x \right)$.  Therefore, applying Theorem \ref{BMO} (or Theorem \ref{BMOQ} ) to the function $w$, we obtain $u \in C^{2,\alpha}(x_0)$ . 
\end{proof}

\subsection{Estimates for more general equations by perturbation}
The results obtained so far provide us with $C^{2,\alpha}$ estimates under weak convexity assumptions when the variable $x$ is frozen, that is, when the operator $F$ does not depend on $x$.   It seems plausible that the techniques in this paper can be modified to yield $C^{2,\alpha}$ regularity results for transmission  problems of the type
\begin{equation}\label{CGS}
  \left\{
	\begin{array}{rclcl}
		F(D^2 u,x) &=& f^{\pm} & \mbox{in} & \Omega^{\pm},\\
		u^{+}_{\nu} - u^{-}_{\nu}&=& g & \mbox{on} & \Gamma.
	\end{array}
	\right.
\end{equation}

As it turns out, this is the case. The starting point relies on the function
$$
 \beta_F(x,x_0) \colon= \sup_{M \in \textrm{Sym}(n)} \frac{F(M,x)-F(M,x_0)}{1+\|M\|},
$$
which measures the oscillation of $F$ in $x$ near the $x_0$. 

\subsection{First case}
 In the first main result, we realize that everything follows the same way as long as we modify Lemma \ref{Lema1} and \ref{Lema2}, since Lemmas \ref{LSS2} and \ref{L02} are consequences of them. Notice that the proof of Lemma \ref{Lema1} is based on the ABP estimate, existence and uniqueness of viscosity solutions and compactness argument. The proof of Theorem 2.1 in \cite{MPA} can readily be modified to yield ABP-type estimates, existence and uniqueness results for operators of the form  $F=F(M,x)$.  Regarding the proof of Lemma \ref{Lema1}, if we include the additional hypothesis $\|\beta\|_{L^{n}(\Omega^{\pm})} \le \delta$ in \eqref{Estima}, we have that the previous techniques can be modified to yield to proof Lemma \ref{Lema1}, including the assumption $\|\beta_j\|_{L^n(\Omega^{\pm}_j)} \to 0$ as $j \to + \infty$ in Lemma \ref{Est}.  To conclude, Lemma \ref{Lema2} follows from Lemma \ref{Lema1}.
 
 \subsection{Second case}
For quasiconvex/quaseconcave operators, we do not need to modify Theorem \ref{flat}.  In Lemma \ref{L1}, it is possible to prove that  given $\epsilon>0$, there exists $\delta>0$ such  that if \eqref{Hip} holds and $\|\beta\|_{L^n(\Omega^{\pm})} \le \delta$, then any two viscosity solutions $u$ and $v$ of \eqref{CGS} and
$$
\left\{
	\begin{array}{rclcl}
		F(D^2 v,0) &=& 0 & \mbox{in} & B^{\pm}_{3/4},\\
		v^{+}_{x_n} - v^{-}_{x_n}&=& D' \cdot x + b & \mbox{on} & B_{3/4} \cap \{x_n=0\}
	\end{array}
	\right.
$$
respectively, satisfy
$$
\|u-v\|_{L^{\infty}(B_{3/4})} \le \epsilon.
$$
As in the previous case,  the proof of Lemmas \ref{L1} is based on the ABP estimates and $C^{\alpha_1}$ regularity of the solution.  It is worth noting that Lemma \ref{L_1} and Lemma \ref{L_2} are consequences of  Lemma \ref{L1}. 

Thus,  as a product of the ideas above, we can extend to more general operators of the type $F=F(M,x)$. In what follows we describe precisely such results.
 
 \begin{theorem}\label{Obs2}
 \end{theorem}
\begin{enumerate}
\item \textbf{(First Case)} Fixe any $\alpha \in (0,1)$. Assume that $g \in C^{1,\alpha}(\Gamma)$ and $f^{\pm} \in C^{0,\alpha}(\overline{\Omega}^{\pm})$ with $f^+ = f^{-}$ on $\Gamma$ and for any $x_0 \in \overline{\Omega^{\pm}_{1/2}}$,
 \begin{equation} \label{OC}
\left( \intav{B_r(x_0) \cap \Omega^{\pm}} \beta^n(x,x_0) dx\right)^{1/n} \le \mathfrak{c} \cdot r^{\alpha},
 \end{equation}
for $r$ small and $\mathfrak{c}>0$. Then there exists $\delta>0$ such that if \eqref{Cor} is satisfied then any bounded viscosity solution $u$ to the problem \eqref{CGS} satisfies $u \in C^{2,\alpha}(x_0)$ with estimate
$$
\|u^{\pm}\|_{C^{2,\alpha}(\overline{\Omega^{\pm}_{1/2}})} \le C_0  \|\psi\|_{C^{2,\alpha}(\overline{B'_1})}\left( \|u\|_{L^{\infty}(B_1)} +  \|g\|_{C^{1,\alpha}(\Gamma)} + \|f^{\pm}\|_{C^{0,\alpha}(\overline{\Omega^{\pm}})}  \right).
$$
where $C_0>0$ depends only on $n,\Lambda, \lambda$ and $\alpha$.
\item \textbf{(Second Case)} Let $F \in \mathcal{E}(\lambda,\Lambda)$ be a quasiconvex or quasiconcave operator. Fixe any $\alpha \in (0,\alpha_0)$. Assume that $g \in C^{1,\alpha}(\Gamma)$ and $f^{\pm} \in C^{0,\alpha}(\overline{\Omega}^{\pm})$ with $f^+ = f^{-}$ on $\Gamma$ and \eqref{OC}.  Then any bounded viscosity solution $u$ to the problem \eqref{CGS} satisfies $u \in C^{2,\alpha}(\overline{\Omega^{\pm}_{1/2}})$ with the same estimate as in item $(1)$.
\end{enumerate}

As immediate consequence of Theorem \ref{Obs2}-(1) we can recover the classical result of Cordes-Nirenberg.

\begin{corollary}[Cordes-Nirenberg] \label{Isaac}
Fix any $x_0 \in \overline{ \Omega^{\pm}_{1/2}}$. Let $u$ be a viscosity solution to 
$$
\left\{
	\begin{array}{rclcl}
		\displaystyle\sup_{\beta \in \mathcal{B}} \, \inf_{\gamma \in \mathcal{A}} \left( \mathcal{L}_{\gamma,\beta} \, u - h_{\beta \gamma}(x) \right)&=& f^{\pm} & \mbox{in} & \Omega^{\pm},\\
		u^{+}_{\nu} - u^{-}_{\nu}&=& g & \mbox{on} & \Gamma,
	\end{array}
	\right.
	$$
where $\mathcal{L}_{\gamma,\beta} \, u \colon= a^{ij}_{\gamma \beta}(x) \partial_{ij} u$ is a family of uniformly elliptic operators with,  $h_{\gamma \beta} \in C^{0,\alpha}(x_0)$ and $(a^{ij}_{\gamma, \beta} - \delta_{ij}) \in C^{0,\alpha}(x_0)$.  There exists $\delta >0$ such that if $a^{ij}_{\gamma \beta}(x)$ fulfills the condition \eqref{Cor}, then $u \in C^{2,\alpha}(x_0)$.
\end{corollary}

\subsection*{Acknowledgments}

\hspace{0.65cm} G.C. Ricarte have been partially supported by CNPq-Brazil under Grants  No. 304239/2021-6.


\begin{thebibliography}{99}

\bibitem{ASS} Armstrong, Scott N.; Silvestre, Luis E. and Smart, Charles K. \textit{Partial regularity of solutions of fully nonlinear uniformly elliptic equations}.  Comm. Pure Appl. Math. {\bf 65} (2012), no. 8, 1169-1184.



\bibitem{borsuk} Borsuk, M . V.  \textit{Transmission problems for elliptic second-order equations in non- smooth domains.} Frontiers in Mathematics. Birkh\"{a}user/Springer Basel AG, Basel,
2010.

\bibitem{bpu} Bianca, V., Pimentel, E.A., Urbano, J.M. \textit{BMO-regularity for a degenerate transmission problem.} Anal.Math.Phys. 14, 9 (2024)



\bibitem{cscs}Caffarelli, L.A., Soria-Carro, M.  Stinga, P.R. \textit{Regularity for 
Interface Transmission Problems.} Arch Rational Mech Anal 240, 265–294 (2021). https://doi.org/10.1007/s00205-021-01611-0


\bibitem{camp1} Campanato, S., Sul probiema di M. Picone relativo all’equilibrio di un corpo elastico in- castrato, Ricerche di Matem. Vol. 1, 1957, pp. 125-149.

\bibitem{camp2} Campanato, S., Sui Problemi al contorno per aietemi di equazioni differenziali lineari del tipo dell'elasticitá, Annali della Souola Normale Superiore di Pisa, Series 3, Vol. 13, Fasc. 2, 1959.


\bibitem{Caff1} Caffarelli, L.A.
\textit{Interior a priori estimates for solutions of fully nonlinear equations.}
Ann. of Math.(2) 130 (1989), no. 1, 189--213.


\bibitem{dong}  Dong, H.  \textit{A simple proof of regularity for $C^{1,\alpha}$ interface transmission problems.} Ann. Appl. Math., 37(1):22–30, 2021.

\bibitem{Al} Goffi, A. \textit{High-order estimates for fully nonlinear equations under weak concavity assumptions}. J. Math. Pures Appl. (9) 182, 223-252, 2024.

\bibitem{Caff2} Caffarelli, L.A., Cabr\'{e}, 
X. \textit{ Fully nonlinear elliptic equations.} American Mathematical Society, Providence (1995)

\bibitem{CIL} Crandall,M.G.,Ishii,H.,Lions,P.-L. \textit{User’s guide to viscosity solutions of second order partial differential equations.} Bull. Amer. Math. Soc. 27, 1–67 (1992)

 \bibitem{Ev} Evans, Lawrence C. \textit{ Classical solutions of fully nonlinear, convex, second-order elliptic equations}. Comm. Pure Appl. Math., 35(3). 333-363, 1982.
 
 \bibitem{JV1} Da Silva, J.V. ; Santos, M. \textit{Schauder and Calderón-Zygmund type estimates for fully nonlinear parabolic equations under "small ellipticity aperture" and applications}. Preprint. \url{https://arxiv.org/pdf/2311.02524.pdf}

 \bibitem{Kr} Krylov, Nikolai V. \textit{ Boundedly inhomogeneous elliptic and parabolic equations in a domain}. Izv. Akad. Nauk SSSR Ser. Mat., 47(1): 75-108, 1983.
 
 
 
 
 \bibitem{livo} Li, Y.,  Vogelius, M. \textit{Gradient estimates for solutions to divergence
 form elliptic equations with discontinuous coefficients.} Arch. Ration. Mech. Anal.,
 153(2):91–151, 2000.
 
 
 
 
 \bibitem{lions}Lions, J. L.  \textit{Contribution à un problème de M. M. Picone.} Ann. Mat. Pura
 Appl. (4), 41:201–219, 1956.
 
 \bibitem{Iss} R. Isaacs. \textit{Differential games}. A mathematical theory with applications to warfare and pursuit, control and optimization. John Wiley \& Sons, Inc., New York-London-Sydney, 1965.
 
 \bibitem{AFD} A. Friedman. \textit{Differential games}. Wiley-Interscience [A division of John Wiley \& Sons, Inc.], New York-London, 1971. Pure and Applied Mathematics, Vol. XXV.
 
\bibitem{schechter} Schechter, M. \textit{A generalization of the problem of transmission.} Annali della Scuola Normale Superiore di Pisa - Classe di Scienze 14.3 (1960): 207-236. 


\bibitem{Tei} Silves, L., Teixeira, E. \textit{Regularity estimates for fully non linear elliptic equations which are asymptotically convex}. Contributions to Nonlinear Elliptic Equations and Systems, Volume 86 of the series Progr. Nonlinear Differential Equations Appl. pp 425-438.

\bibitem{MPA} Maria Soria-Carro, Pablo Raúl Stinga. \textit{Regularity of Viscosity solutions to fully nonlinear elliptic transmission problems} Advances in Mathematics 435, 1-51 (2023).

\bibitem{Nir} Nirenberg, Louis. \textit{On nonlinear elliptic partial differential equations and H\"{o}lder continuity}. Comm. Pure Appl. Math, 6 (1953), 103-156.\label{Nir}

\bibitem{Tei1} Teixeira, Eduardo V. \textit{Universal moduli of continuity for solutions to fully nonlinear elliptic equations.}  Arch. Rational Mech. Anal.

\bibitem{stampacchia} Stampacchia, G., Su un problema relativo alle equazioni di tipo ellittico del secondo ordine, Ricerche di matem., Vol. 5, 1956, pp. 3-24.


\bibitem{Vla} N. Nadirashvili, S. Vl$\breve{a}$dut, Nonclassical solutions of fully nonlinear elliptic equations, Geom. Funct. Anal. 17 (4) (2007) 1283–1296, \url{http://dx.doi.org/10.1007/s00039-007-0626-7}.

\bibitem{Vla1} N. Nadirashvili, S. Vl$\breve{a}$dut, Singular viscosity solutions to fully nonlinear elliptic equations, J. Math. Pures Appl. (9) 89 (2) (2008) 107–113, \url{http://dx.doi.org/10.1016/j.matpur.2007.10.004}.

\bibitem{Vla2} N. Nadirashvili, S. Vl$\breve{a}$dut, Singular solution to special Lagrangian equations, Ann. Inst. H. Poincaré Anal. Non Linéaire 27 (5) (2010) 1179–1188, \url{http://dx.doi.org/10.1016/j.anihpc.2010.05.001}.

\bibitem{Picone}M. Picone. \textit{Nuovi indirizzi di ricerca nella teoria e nel calcolo delle soluzioni di talune equazioni lineari alle derivate parziali della fisica-matematica.} Ann. Scuola Norm. Super. Pisa Cl. Sci. (2), 5(3-4):213–288, 1936.



\bibitem{Tru} Trudinger, Neil S. \textit{On regularity and existence of viscosity solutions of nonlinear second order, elliptic equations.} In \textit{Partial differential equations and the calculus of variations, Vol. II, Progr. Nonlinear Differential Equations Appl.,} Pages 939-957. Birkh\"{a}user Boston, Boston, MA. 1989.\label{Tru}

\bibitem{Tru1} Trudinger, Neil S. \textit{H\"{o}lder gradient estimates for fully nonlinear elliptic equations.} Proc. Roy. Soc. Edinburgh Sect. A, 108(1-2):57–65, 1988.

\bibitem{WN} Wu, D. and Niu, P., \textit{Interior pointwise $C^{2,\alpha}$ regularity for fully nonlinear elliptic equations.} Nonlinear Anal. 227 (2023), Paper No. 113159, 9 pp.

\bibitem{da Silva1}da Silva, J.V., dos Prazeres, D. \textit{Schauder type estimates for “flat” viscosity solutions to non-convex fully nonlinear parabolic equations and applications.} Potential Anal. 50 (2019), no. 2, 149-170.

\bibitem{da Silva2}da Silva, J.V., Teixeira, E.V. \textit{Sharp regularity estimates for second order fully nonlinear parabolic equations.} Math. Ann. 369(3–4), 1623–1648 (2017).

\bibitem{Giga} Giga, Y. and Sato, M.-H. \textit{On semicontinuous solutions for general Hamilton-Jacobi equations.} Comm. Partial Differential Equations 26 (2001), 813-839.

\bibitem{RT} Ricarte G.C. and  Teixeira, E.V.
\textit{Fully nonlinear singularly perturbed equations and asymptotic free boundary}.
J. Funct. Anal., vol. 261, Issue 6, 2011, 1624-1673.



\end{thebibliography}
\end{document}